\documentclass[12pt,reqno]{amsart} 
\usepackage{amssymb}
\usepackage{mathrsfs}
\usepackage{enumerate}
\usepackage[all]{xy}
\usepackage[usenames]{color}
\usepackage[colorlinks=true, citecolor=blue, linkcolor=red]{hyperref}

\newcommand{\xxx}[2][k]{#2_1\times\cdots\times #2_{#1}}

\let\reff=\ref  \let\eqreff=\eqref
\newcommand{\nnewpage}{\bigskip}

\renewenvironment{enumerate}[1][]
{\begin{enumerat}[#1]\setlength{\itemsep}{6pt}}{\end{enumerat}}

\newenvironment{enuma}{\begin{enumerate}[{\rm(a) }]}{\end{enumerate}}

\renewenvironment{itemize}
{\begin{itemiz}\setlength{\itemsep}{6pt} \setlength{\itemindent}{-10pt} 
}{\end{itemiz}}

\definecolor{darkgreen}{rgb}{0,0,0}
\definecolor{bluegreen}{rgb}{0,0.2,0.8}
\definecolor{darkred}{rgb}{0.8,0,0}
\definecolor{newercolor}{rgb}{0.2,0,1}
\definecolor{darkyellow}{rgb}{0.7,0.7,0}
\definecolor{orange}{rgb}{1.0,0.5,0}
\definecolor{darkorange}{rgb}{0.8,0.4,0}

\long\def\new#1{{\color{darkgreen}#1}}
\newenvironment{New}{\begingroup\color{darkgreen}}{\endgroup}

\numberwithin{table}{section}

\newcommand{\boldd}[1]{{\mathversion{bold}\textbf{#1}}}

\newlength{\short}
\newcommand{\4}[1]{\widebar{#1}}
\newcommand{\5}[1]{\widehat{#1}}

\let\8=\scriptstyle

\def\pair[#1,#2]{[\hskip-1.5pt[#1,#2]\hskip-1.5pt]}

\SelectTips{cm}{10} \UseTips   

\let\oldcirc=\circ
\renewcommand{\circ}{\mathchoice
    {\mathbin{\scriptstyle\oldcirc}}{\mathbin{\scriptstyle\oldcirc}}
    {\mathbin{\scriptscriptstyle\oldcirc}}
    {\mathbin{\scriptscriptstyle\oldcirc}}}

\def\beq#1\eeq{\begin{equation*}#1\end{equation*}}
\def\beqq#1\eeqq{\begin{equation}#1\end{equation}}

\numberwithin{equation}{section}

\newtheorem{Thm}{Theorem}[section]
\newtheorem{Prop}[Thm]{Proposition}
\newtheorem{Cor}[Thm]{Corollary}
\newtheorem{Lem}[Thm]{Lemma}

\newtheorem{Ex}[Thm]{Example}
\newtheorem{Rmk}[Thm]{Remark}

\newtheorem{Th}{Theorem}

\theoremstyle{definition}
\newtheorem{Defi}[Thm]{Definition}
\newtheorem{DefNot}[Thm]{Definition-Notation}

\newcommand{\widebar}[1]
      {\overset{{\mskip3mu\leaders\hrule height0.4pt\hfill\mskip3mu}}{#1}
      \vphantom{#1}}

\newcounter{let} 
\loop\stepcounter{let}
\expandafter\edef\csname cal\alph{let}\endcsname%
{\noexpand\mathcal{\Alph{let}}}
\ifnum\thelet<26\repeat

\loop\stepcounter{let}
\expandafter\edef\csname scr\alph{let}\endcsname%
{\noexpand\mathscr{\Alph{let}}}
\ifnum\thelet<26\repeat

\newcommand{\tdef}[2][]{\expandafter\newcommand\csname#2\endcsname%
{#1\textup{#2}}}
\tdef{Iso}   \tdef{Aut}    \tdef{Out}    \tdef{Inn}    \tdef{Hom}
\tdef{End}   \tdef{Inj}    \tdef{map}    \tdef{Ker}    \tdef{Ob}
\tdef{Mor}   \tdef{Res}    \tdef{Id}     \tdef{Fr}     \tdef{Spin} 
\tdef{rk}    \tdef{conj}   \tdef{incl}   \tdef{proj}   \tdef{diag} 
\tdef{trf}   \tdef{Sol}    \tdef{He}     \tdef{Sz}     \tdef{cj}
\tdef{Rep}   \tdef{pr}    \tdef{Inndiag} \tdef{Outdiag}  \tdef{expt}
\tdef{supp}  \tdef{Isom}   \tdef{ord}    \tdef{Coker}   \tdef{Tr}
\tdef[_]{typ} \tdef[^]{op} \tdef[^]{ab}   \tdef{lcm}  \tdef{McL}
\tdef{restr}

\newcommand{\fdef}[1]{\expandafter\newcommand\csname#1\endcsname%
{\mathfrak{#1}}}
\fdef{X}  \fdef{red}  \fdef{foc}  \fdef{hyp}  \fdef{Lie} \fdef{Y}

\newcommand{\bbdef}[1]{\expandafter\newcommand%
\csname#1\endcsname{\mathbb{#1}}}
\bbdef{C} \bbdef{F} \bbdef{R} \bbdef{Z} \bbdef{N} \bbdef{Q} \bbdef{K}

\newcommand{\itdef}[1]{\expandafter\newcommand\csname#1\endcsname%
{\textit{#1}}}
\itdef{PSL}  \itdef{PSU}  \itdef{SL}  \itdef{SU}  \itdef{GL} \itdef{GU}
\itdef{Sp}   \itdef{PSp} \itdef{PSO} \itdef{SO}   \itdef{SD} \itdef{PGU} 
\itdef{PGL}  \itdef{Co}  \itdef{Fi}  \itdef{GO}

\newcommand{\sminus}{\smallsetminus}
\newcommand{\lie}[3]{\def\test{#2}\def\tst{G}\ifx\test\tst{{}^{#1}#2_{#3}}
\else{{}^{#1}\!#2_{#3}}\fi}
\renewcommand{\*}{\,\lower6pt\hbox{\Large{\textup{*}}}\,}
\newcommand{\syl}[2]{\textup{Syl}_{#1}(#2)}
\newcommand{\sylp}[1]{\syl{p}{#1}}

\renewcommand{\Im}{\textup{Im}}
\newcommand{\autf}{\Aut_{\calf}}

\newcommand{\outf}{\Out_{\calf}}

\newcommand{\homf}{\Hom_{\calf}}
\newcommand{\isof}{\Iso_{\calf}}
\newcommand{\defeq}{\overset{\textup{def}}{=}}

\newcommand{\mxfoura}[8]{\left(\begin{smallmatrix}#1&#2&#3&#4\\#5&#6&#7&#8}
\newcommand{\mxfourb}[8]{\\#1&#2&#3&#4\\#5&#6&#7&#8\end{smallmatrix}\right)}

\let\emptyset=\varnothing
\renewcommand{\:}{\colon}

\newcommand{\nsg}{\trianglelefteq}

\newcommand{\til}[1]{\widetilde{#1}}
\let\too=\longrightarrow

\newcommand{\gen}[1]{{\langle}#1{\rangle}}

\newcommand{\longleft}[1]{\;{\leftarrow%
\count255=0 \loop \mathrel{\mkern-6mu}%
    \relbar\advance\count255 by1\ifnum\count255<#1\repeat}\;}
\newcommand{\longright}[1]{\;{\count255=0 \loop \relbar\mathrel{\mkern-6mu}%
    \advance\count255 by1\ifnum\count255<#1\repeat\rightarrow}\;}
\newcommand{\Right}[2]{\overset{#2}{\longright#1}}
\newcommand{\RIGHT}[3]{\mathrel{\mathop{\kern0pt\longright#1}
        \limits^{#2}_{#3}}}

\newcommand{\LEFT}[3]{\mathrel{\mathop{\kern0pt\longleft#1}\limits^{#2}_{#3}}
}
\newcommand{\dRIGHT}[3]{\mathrel{%
   \mathop{\vcenter{\baselineskip=0pt\hbox{$\kern0pt\longright#1$}%
   \hbox{$\kern0pt\longright#1$}}}\limits^{#2}_{#3}}}
\newcommand{\LRIGHT}[3]{\mathrel{%
   \mathop{\vcenter{\baselineskip=0pt\hbox{$\kern0pt\longleft#1$}%
   \hbox{$\kern0pt\longright#1$}}}\limits^{#2}_{#3}}}
\newcommand{\RLEFT}[3]{\mathrel{%
   \mathop{\vcenter{\baselineskip=0pt\hbox{$\kern0pt\longright#1$}%
   \hbox{$\kern0pt\longleft#1$}}}\limits^{#2}_{#3}}}
\newcommand{\onto}[1]{\;{\count255=0 \loop \relbar\mathrel{\mkern-6mu}%
    \advance\count255 by1
    \ifnum\count255<#1 \repeat \twoheadrightarrow}\;}

\begin{document}

\title{A Krull-Remak-Schmidt theorem for fusion systems}

\author{Bob Oliver}
\address{Universit\'e Sorbonne Paris Nord, LAGA, UMR 7539 du CNRS, 
99, Av. J.-B. Cl\'ement, 93430 Villetaneuse, France.}
\email{bobol@math.univ-paris13.fr}
\thanks{B. Oliver is partially supported by UMR 7539 of the CNRS}



\begin{abstract} 
We prove that the factorization of a saturated fusion system over a 
discrete $p$-toral group as a product of indecomposable subsystems is 
unique up to normal automorphisms of the fusion system and permutations 
of the factors. In particular, if the fusion system has trivial center, 
or if its focal subgroup is the entire Sylow group, then this 
factorization is unique (up to the ordering of the factors). This result 
was motivated by questions about automorphism groups of products of 
fusion systems.
\end{abstract}

\maketitle

\bigskip


Let $\Z/p^\infty$ denote the union of an ascending sequence 
$\Z/p\le\Z/p^2\le\Z/p^3\le\cdots$ of finite cyclic $p$-groups. A discrete 
$p$-toral group is an extension of a group isomorphic to $(\Z/p^\infty)^r$ 
(some $r\ge0$) by a finite $p$-group. A saturated fusion system $\calf$ 
over a discrete $p$-toral group $S$ is a category whose objects are the 
subgroups of $S$, whose morphisms are injective homomorphisms between the 
objects, and which satisfies certain axioms first formulated by Puig 
\cite{Puig} when $S$ is a finite $p$-group, and by Broto, Levi, and this 
author \cite{BLO3} in the more general case. 

For each compact Lie group $G$ and each prime $p$, there is a saturated 
fusion system over a maximal discrete $p$-toral subgroup $S\le G$ that 
encodes the $G$-conjugacy relations between subgroups of $S$ (see 
\cite[\S\,9]{BLO3}). Likewise, each torsion linear group in characteristic 
different from $p$ (i.e., each subgroup $G\le\GL_n(K)$ such that $n\ge1$, 
$K$ is a field with $\textup{char}(K)\ne p$, and all elements of $G$ have 
finite order) has a maximal discrete $p$-toral subgroup $S$ unique up to 
conjugacy, and a saturated fusion system over $S$ that encodes $G$-conjugacy 
relations among subgroups of $S$ \cite[Theorem 8.10]{BLO3}. 

The Krull-Remak-Schmidt theorem for groups says, in the case of finite 
groups, that for any two factorizations of $G$ as a product of 
indecomposable subgroups, there is a normal automorphism of $G$ that 
sends the one to the other. Here, $\alpha\in\Aut(G)$ is normal if it 
commutes with all inner automorphisms; equivalently, if $\alpha$ is the 
identity on $[G,G]$ and induces the identity on $G/Z(G)$. We refer to 
\cite[Theorem 2.4.8]{Sz1} or \cite[Satz I.12.3]{Huppert} for the 
complete (much stronger) theorem. 

By analogy, if $\alpha$ is an automorphism of a saturated fusion system 
$\calf$ over a discrete $p$-toral group $S$ (see Definition 
\reff{d:Mor(E,F)}), $\alpha$ is normal if $\alpha|_{\foc(\calf)}=\Id$ and 
$[\alpha,S]\le Z(\calf)$. See Definitions \reff{d:Z+foc} and 
\reff{d:normal.end.} and Lemma \reff{l:[f,S]<Z(F)}(a) for more details (and 
the more general definition of normal endomorphisms). In these terms, our 
main theorem is formulated as follows.

\begin{Th} \label{ThA}
Let $\calf$ be a saturated fusion system over a discrete $p$-toral group 
$S$. Then there exist indecomposable fusion subsystems 
$\cale_1,\dots,\cale_k\le\calf$ ($k\ge1$) such that 
	\[ \calf = \xxx\cale. \] 
If $\calf=\cale_1^*\times\dots\times\cale_m^*$ is another such 
factorization, then 
$k=m$, and there is a normal automorphism $\alpha\in\Aut(\calf)$ and a 
permutation $\sigma\in\Sigma_k$ such that 
$\alpha(\cale_i)=\cale^*_{\sigma(i)}$ for each $i$.
\end{Th}

The existence of a factorization into a product of indecomposables is 
elementary, and is shown in Proposition \reff{p:exists.fact.}. The 
uniqueness part of Theorem \reff{ThA} is a special case of Theorem 
\reff{t:KRS} (the case where $\Omega=1$), and our proof of that theorem is 
adapted directly from that in \cite{Sz1} of the Krull-Remak-Schmidt theorem 
for groups. It is mostly a question of finding good definitions and 
properties of commuting fusion subsystems and normal endomorphisms of 
fusion systems; once this has been done it is straightforward to translate 
the proof of Theorem 2.4.8 in \cite{Sz1} into this situation.

As a special case, when $Z(\calf)=1$ or $\foc(\calf)=S$, Theorem \reff{ThA} 
says that the factorization of $\calf$ is unique: that $\calf$ is the 
product of all of its indecomposable direct factors. As one consequence of 
this (Corollary \reff{c:Aut(F1xF2)}), if $\calf=\xxx\cale$ where the 
$\cale_i$ are indecomposable (and $Z(\calf)=1$ or $\foc(\calf)=S$), then 
$\Aut(\calf)$ is a semidirect product of $\prod_{i=1}^k\Aut(\cale_i)$ with 
a certain subgroup of $\Sigma_k$. 

This work was originally motivated by questions about automorphisms of 
products of fusion systems that arose during joint work with Carles Broto, 
Jesper M{\o}ller, and Albert Ruiz \cite{BMOR}. It turned out that a special 
case of Theorem \reff{ThA} was sufficient for our purposes in that paper: a 
case which had been proven earlier in \cite[Proposition 3.6]{AOV1}. But 
this led to the question of whether a stronger result of that type might 
also be true. 

\begin{New} 
We have been asked whether (and how easily) Theorem \ref{ThA}, when 
restricted to fusion systems realized by finite groups, can be proven as a 
consequence of the Krull-Remak-Schmidt theorem for finite groups. When 
$p=2$ and $O^{2'}(\calf)=\calf$, this can be done using a theorem of 
Goldschmidt on strongly closed $2$-subgroups of a finite group 
\cite[Theorem A]{Goldschmidt}. When $p=2$ but $O^{2'}(\calf)<\calf$, it can 
be shown using similar, but much more complicated arguments. When $p$ is an 
odd prime, any such argument probably requires the classification of 
finite simple groups. See the end of Section 5 for a more detailed 
discussion about this question. 
\end{New}

We have tried to write this while keeping in mind those readers who are 
interested only in the case of fusion systems over finite $p$-groups. 
For this reason, as far as possible, the extra complications that arise 
in the infinite case have been put into Section 1, which can easily be 
skipped by those interested only in the finite case and familiar with 
the basic definitions. Morphisms and commuting subsystems 
of fusion systems are defined and studied in Section 2, sums of 
endomorphisms in Section 3, and normal endomorphisms in Section 4. The 
main theorem and two corollaries are proven in Section 5.

\new{We take the opportunity here to thank the referee for carefully 
reading this paper, and for making several very helpful suggestions for 
improvements.}

\noindent{\textbf{Notation and conventions:}} Composition of functions and 
functors is always from right to left. When $G$ is a group and $P,Q\le G$, 
we let $\Hom_G(P,Q)$ denote the set of (injective) homomorphisms from $P$ 
to $Q$ induced by conjugation in $G$, and set 
$\Aut_G(P)=\Hom_G(P,P)\cap\Aut(P)$. 


\nnewpage

\section{\texorpdfstring{Fusion systems over discrete $p$-toral groups}
{Fusion systems over discrete p-toral groups}}

In this section, we collect some results that are needed mostly when 
handling fusion systems over \emph{infinite} discrete $p$-toral groups. So 
readers who are already familiar with fusion systems over finite 
$p$-groups and only interested in that case can easily skip it.

As defined in the introduction, $\Z/p^\infty$ denotes the union of the 
chain of cyclic $p$-groups $\Z/p<\Z/p^2<\Z/p^3<\cdots$. It can also 
be identified with the quotient group $\Z[\frac1p]/\Z$, or with the 
group of complex roots of unity of $p$-power order. 

\begin{Defi} \label{d:discr.p-tor.}
A \emph{discrete $p$-toral group} is a group $S$, with normal subgroup 
$S_0\nsg{}S$, such that $S_0$ is isomorphic to a finite product of copies 
of $\Z/p^\infty$ and $S/S_0$ is a finite $p$-group. The subgroup $S_0$ will 
be called the \emph{identity component} of $S$, and $S$ will be called 
\emph{connected} if $S=S_0$.  Define $|S|=(\rk(S_0),|S/S_0|)$, where 
$\rk(S_0)=k$ if $S_0\cong(\Z/p^\infty)^k$. 
\end{Defi}

The identity component $S_0$ of a discrete $p$-toral group $S$ is 
characterized as the subset of all infinitely $p$-divisible elements in 
$S$, and also as the minimal subgroup of finite index in $S$. So $|S|$ 
depends only on $S$ itself as a discrete group. We regard the order of a 
discrete $p$-toral group as an element of $\N^2$ with the lexicographical 
ordering; i.e., $|S|\le|S^*|$ if and only if either $\rk(S)<\rk(S^*)$, or 
$\rk(S)=\rk(S^*)$ and $|S/S_0|\le|S^*/S^*_0|$.  Note that $S^*\le{}S$ 
implies $|S^*|\le|S|$, with equality only if $S^*=S$. 

Recall that a group is \emph{artinian} if each descending sequence of 
subgroups of $S$ becomes constant. Discrete $p$-toral groups can be 
characterized as follows:

\begin{Lem}[{\cite[Proposition 1.2]{BLO3}}] \label{l:artinian}
A group is discrete $p$-toral if and only if it is artinian, and every 
finitely generated subgroup is a finite $p$-group.
\end{Lem}

In particular, each subgroup or quotient group of a discrete $p$-toral 
group is again discrete $p$-toral, and each extension of one discrete 
$p$-toral group by another is discrete $p$-toral \cite[Lemma 1.3]{BLO3}. If 
$Q\le P$ is a pair of discrete $p$-toral groups, then 
$\Out_P(Q)\defeq\Aut_P(Q)/\Inn(Q)$ ($\cong N_P(Q)/QC_P(Q)$) is a finite 
$p$-group (see \cite[Proposition 1.5(c)]{BLO3}). 

Before defining fusion systems, we prove two technical results about 
discrete $p$-toral groups that are elementary or well known for finite 
$p$-groups. The first deals with complications that arise because discrete 
$p$-toral groups need not be nilpotent.

\begin{Lem} \label{l:Ln(S)}
Let $S$ be a discrete $p$-toral group, and let $S_0$ be its identity 
component. Define inductively subgroups $\til{Z}_n(S)\nsg S$, for $n\ge0$, by 
setting $\til{Z}_0(S)=1$, $\til{Z}_1(S)=\Omega_1(Z(S))$, and 
$\til{Z}_n(S)/\til{Z}_{n-1}(S)=\Omega_1(Z(S/\til{Z}_{n-1}(S)))$. Set 
$\til{Z}_\infty(S)=\bigcup_{n=1}^\infty \til{Z}_n(S)$. Then 
\begin{enuma} 

\item $\til{Z}_\infty(S)\ge S_0$ and $C_S(\til{Z}_\infty(S))\le \til{Z}_\infty(S)$; and 

\item if $\alpha\in\Aut(S)$ has finite order prime to $p$, and 
$[\alpha,\til{Z}_n(S)]\le \til{Z}_{n-1}(S)$ for each $n\ge1$, then $\alpha=\Id_S$.

\end{enuma}
\end{Lem}

\begin{proof} \textbf{(a) } For each pair of finite subgroups $P,Q\le S_0$, 
both normal in $S$ and such that $P\nleq Q$, we have $(PQ/Q)\cap 
\til{Z}_1(S/Q)\ge \Omega_1(C_{PQ/Q}(S/S_0))$, where 
$\Omega_1(C_{PQ/Q}(S/S_0))\ne1$ since $PQ/Q$ and $S/S_0$ are both finite 
$p$-groups and $PQ/Q\ne1$. When $Q=\til{Z}_{n-1}(S)$ for $n\ge1$, this says 
that $P\cap \til{Z}_n(S)>P\cap \til{Z}_{n-1}(S)$ whenever 
$\til{Z}_{n-1}(S)\ngeq P$, and hence that $\til{Z}_m(S)\ge P$ for $m$ 
sufficiently large. In particular, $\til{Z}_\infty(S)\ge\Omega_n(S_0)$ for 
each $n$, and so $\til{Z}_\infty(S)\ge S_0$. 

Set $T=\til{Z}_\infty(S)$ and $U=C_S(T)$ for short, and assume $U\nleq T$. 
Then $1\ne UT/T\nsg S/T$ where $S/T$ is a finite $p$-group \new{(recall 
$T=\til{Z}_\infty(S)\ge S_0$)}, so $(UT/T)\cap\Omega_1(Z(S/T))\ne1$. 
\new{In other words, there is $x\in S\sminus T$ such that 
	\[ [x,T]=1, \qquad [x,S]\le T, \qquad\textup{and}\qquad x^p\in T. \] 
Since} $[x,T]=1$ and $S/T$ is finite (and since $T=\bigcup_{m=1}^\infty 
\til{Z}_m(S)$), there is $n\ge1$ such that $[x,S]\le \til{Z}_n(S)$ and 
$x^p\in \til{Z}_n(S)$. Then $x\in \til{Z}_{n+1}(S)\le T$, contradicting our 
assumption \new{that $x\notin T$}. We conclude that $U\le T$; i.e., that 
$C_S(\til{Z}_\infty(S))\le \til{Z}_\infty(S)$.

\smallskip

\noindent\textbf{(b) } Assume $\alpha\in\Aut(S)$ has finite order prime to 
$p$ and induces the identity on each quotient group $\til{Z}_n(S)/\til{Z}_{n-1}(S)$ 
(all $n\ge1$). Since each of those quotients is a finite $p$-group, 
$\alpha|_{\til{Z}_n(S)}=\Id$ for each $n$ by \cite[Theorem 5.3.2]{Gorenstein}, 
and hence $\alpha|_{\til{Z}_\infty(S)}=\Id$. So by \cite[Lemma 1.2]{OV2}, and 
since $C_S(\til{Z}_\infty(S))\le \til{Z}_\infty(S)$ by (a), the class 
$[\alpha]\in\Out(S)$ is in the image of a certain injective homomorphism 
$\eta\:H^1(S/\til{Z}_\infty(S);Z(\til{Z}_\infty(S)))\too\Out(S)$. Each element in 
$H^1(S/\til{Z}_\infty(S);Z(\til{Z}_\infty(S))$ has order dividing $|S/\til{Z}_\infty(S)|$ 
(see, e.g., Corollary 2 to \cite[Theorem 2.7.26]{Sz1}), and hence 
$[\alpha]=1$ and $\alpha\in\Inn(S)$. But then $\alpha=1$, since it has 
order prime to $p$ while $\Inn(S)\cong S/Z(S)$ is discrete $p$-toral. 
\end{proof}

The following generalization of nilpotent endomorphisms will be needed. 

\begin{Defi} \label{d:loc.nilp.}
Let $S$ be a discrete $p$-toral group. An endomorphism $f\in\End(S)$ is 
\emph{locally nilpotent} if for each $x\in S$, there is $n\ge1$ such that 
$f^n(x)=1$. 
\end{Defi}

Thus $f\in\End(S)$ is locally nilpotent if and only if 
$S=\bigcup_{n=1}^\infty\Ker(f^n)$. 

If $S$ is finite, then clearly all locally nilpotent endomorphisms are 
nilpotent. As a simple example of an endomorphism that is locally 
nilpotent but not nilpotent, let $S$ be any discrete $p$-toral group 
that is abelian and infinite, and set $f=(x\mapsto x^p)\in\End(S)$.

\begin{Lem} \label{l:End(S)-2}
Let $S$ be an abelian discrete $p$-toral group, and assume $f\in\End(S)$ is 
surjective. Then there are unique subgroups $T,U\le S$ such that $S=T\times 
U$, $f|_T\in\Aut(T)$, $U$ is connected, and $f|_U\in\End(U)$ is locally 
nilpotent. 
\end{Lem}

\begin{proof} For each $n\ge1$, set $S_n=\Omega_n(S)$ and 
$f_n=f|_{S_n}$. Then $S_n$ is a finite abelian $p$-group, 
$\{\Im(f_n^i)\}_{i=1}^\infty$ is a decreasing sequence of subgroups of 
$S_n$, and $\{\Ker(f_n^i)\}_{i=1}^\infty$ is an increasing sequence. 
Set $T_n=\bigcap_{i=1}^\infty\Im(f_n^i)$ and 
$U_n=\bigcup_{i=1}^\infty\Ker(f_n^i)$. Since $S_n$ is finite, there is 
$k\ge1$ such that $T_n=\Im(f_n^k)$ and $U_n=\Ker(f_n^k)$, and so 
$|T_n||U_n|=|S_n|$. Also, $f_n(T_n)=T_n$ (so $f_n\in\Aut(T_n)$), 
$f_n|_{U_n}$ is a nilpotent endomorphism of $U_n$, and hence $T_n\cap 
U_n=1$ and $S_n=T_n\times U_n$.

From these properties, we see that $T_n\le T_{n+1}$ and $U_n\le 
U_{n+1}$ for all $n$. Set $U=\bigcup_{i=1}^\infty U_n$ and 
$T=\bigcup_{n=1}^\infty T_n$. Then $S=T\times U$, $f|_T\in\Aut(T)$, and 
$f|_U\in\End(U)$ is locally nilpotent. Note that 
$U=\bigcup_{i=1}^\infty\Ker(f^i)$.

Assume $S=T^*\times U^*$ is a second factorization, where 
$f|_{T^*}\in\Aut(T^*)$ and $f|_{U^*}$ is a locally nilpotent 
endomorphism. Then $U^*\le\bigcup_{i=1}^\infty\Ker(f^i)=U$. For each 
$n\ge1$, $f|_{\Omega_n(T^*)}\in\Aut(\Omega_n(T^*))$ since $f|_{T^*}$ is 
an automorphism, so 
$\Omega_n(T^*)\le\bigcap_{i=1}^\infty\Im(f_n^i)=T_n$, and hence 
$T^*=\bigcup_{i=1}^\infty\Omega_n(T^*)\le T$. Then $T^*=T$ and $U^*=U$ 
since $T^*\times U=T\times U$, proving that the decomposition is unique.

It remains to show that $U$ is connected; i.e., that 
$U\cong(\Z/p^\infty)^r$ for some $r\ge0$. Let $U_0\le U$ be the 
identity component of $U$. Set $\psi=f|_U\in\End(U)$ for short. Since 
$\psi$ is surjective, $U/\Ker(\psi^i)\cong U$ for each $i\ge1$. 
Since $U$ is the union of the $\Ker(\psi^i)$ and $U_0$ has finite index in 
$U$, there is $k\ge1$ such that $U_0\Ker(\psi^k)=U$. So $U/\Ker(\psi^k)\cong 
U$ is a quotient group of $U_0$ and hence connected. 
\end{proof}

We next consider fusion systems over discrete $p$-toral groups. 

\begin{Defi}[{\cite[Definitions 2.1--2.2]{BLO3}}] \label{d:sat.f.s.}
Fix a discrete $p$-toral group $S$. 
\begin{enuma}

\item A \emph{fusion system} $\calf$ over $S$ is a 
category whose objects are the subgroups of $S$, and whose morphism sets 
$\homf(P,Q)$ are such that \medskip
\begin{itemize}
\item $\Hom_S(P,Q)\subseteq\homf(P,Q)\subseteq\Inj(P,Q)$ for all $P,Q\le 
S$; and 

\item every morphism in $\calf$ factors as an isomorphism in $\calf$
followed by an inclusion.
\end{itemize} \medskip
Two subgroups $P,P'\le{}S$ are \emph{$\calf$-conjugate} if 
$\isof(P,P')\ne\emptyset$, and two elements $x,y\in S$ are 
$\calf$-conjugate if there is $\varphi\in\homf(\gen{x},\gen{y})$ such that 
$\varphi(x)=y$. The $\calf$-conjugacy classes of $P\le S$ and $x\in S$ are 
denoted $P^\calf$ and $x^\calf$, respectively.

\item A subgroup $P\le S$ is \emph{fully automized} in $\calf$ if the 
index of $\Aut_S(P)$ in $\autf(P)$ is finite and prime to $p$.

\item A subgroup $P\le S$ is \emph{receptive} in $\calf$ if the 
following holds: for each $Q\in P^\calf$ and each 
$\varphi\in\Iso_\calf(Q,P)$, if we set 
	\[ N_\varphi = N_\varphi^\calf = 
	\{ g\in{}N_S(Q) \,|\, \varphi c_g\varphi^{-1} \in \Aut_S(P) \}, \]
then $\varphi$ extends to a homomorphism $\4\varphi\in\homf(N_\varphi,S)$. 

\item $\calf$ is a \emph{saturated fusion system} if the following
two conditions hold: 
\medskip

\begin{itemize} 

\item For each $P\le S$, there is $R\in P^\calf$ such that $R$ is fully 
automized and receptive in $\calf$.

\item (Continuity axiom) If $P_1\le{}P_2\le{}P_3\le\cdots$ is an increasing 
sequence of subgroups of $S$, with $P_\infty=\bigcup_{n=1}^\infty{}P_n$, 
and if $\varphi\in\Hom(P_\infty,S)$ is any homomorphism such that 
$\varphi|_{P_n}\in\homf(P_n,S)$ for all $n$, then 
$\varphi\in\homf(P_\infty,S)$.

\end{itemize}
\end{enuma}
\end{Defi}

This definition of saturation is different from that given in 
\cite{BLO3}, but is equivalent to it by \cite[Corollary 1.8]{BLO6}. For 
finite $S$, it is the definition used in \cite[\S\,I.2]{AKO}. 

Note that $\outf(P)$ ($=\autf(P)/\Inn(P)$) is finite for each saturated 
fusion system $\calf$ over $S$ and each $P\le S$. If $P$ is fully 
automized, then this follows from the definition and since \new{$\Out_S(P)$} 
is finite (\cite[Proposition 1.5(c)]{BLO3}). Otherwise, there is some 
$R\in P^\calf$ that is fully automized, and $\outf(P)\cong\outf(R)$ is 
finite.

The following additional definitions will be needed.

\begin{Defi} \label{d:Z+foc}
Let $\calf$ be a fusion system over a discrete $p$-toral group $S$. For 
$P\le S$, 
\begin{itemize} 

\item $P$ is \emph{$\calf$-centric} if $C_S(Q)\le Q$ for each $Q\in P^\calf$;

\item $P$ is \emph{$\calf$-radical} if $O_p(\outf(P))=1$; 

\item $P$ is \emph{central} in $\calf$ if each 
$\varphi\in\homf(Q,R)$, for $Q,R\le S$, extends to some 
$\4\varphi\in\homf(QP,RP)$ such that $\4\varphi|_P=\Id_P$; and 

\item $P$ is \emph{strongly closed in $\calf$} if for each $x\in P$, 
$x^\calf\subseteq P$.
\end{itemize}
In addition, 
\begin{itemize} 
\item $Z(\calf)\nsg S$ (the \emph{center} of $\calf$) is the 
subgroup generated by all subgroups $Z\le S$ central in $\calf$; and 

\item $\foc(\calf)=\gen{xy^{-1}\,|\,x,y\in S, ~ y\in x^\calf} \nsg S$ 
(the \emph{focal subgroup} of $\calf$). 

\end{itemize}
\end{Defi}

It follows immediately from the definitions that $Z(\calf)\le Z(S)$ and is 
itself central in $\calf$, and that $\foc(\calf)\ge[S,S]$. 

\begin{Lem} \label{l:Z+foc}
Let $\calf$ be a fusion system over a discrete $p$-toral group $S$. 
Then 
\begin{enuma} 

\item $Z(\calf)\subseteq\bigl\{x\in Z(S)\,\big|\, x^\calf=\{x\}\bigr\}$, 
with equality if $\calf$ is saturated; and 


\item if $P\le S$ is such that $P\le Z(\calf)$ or $P\ge\foc(\calf)$, 
then $P$ is strongly closed in $\calf$.

\end{enuma}
\end{Lem}

\begin{proof} \textbf{(a) } The inclusion is immediate from the definition 
of a central subgroup. The opposite implication (when $\calf$ is saturated) 
is shown in \cite[Lemma I.4.2]{AKO} when $S$ is a finite $p$-group, and the 
same argument applies in the discrete $p$-toral case.

\smallskip

\noindent\textbf{(b) } If $P\le Z(\calf)$, then by (a), $x^\calf=\{x\}$ 
for each $x\in P$, and hence $P$ is strongly closed. If 
$P\ge\foc(\calf)$, then for each $x\in P$ and each $y\in x^\calf$, 
$y=x(x^{-1}y)\in P$ since $x^{-1}y\in\foc(\calf)$, so $P$ is strongly 
closed also in this case.
\end{proof}

We next look at fusion subsystems.

\begin{DefNot} \label{d:f.subsyst.}
Let $\calf$ be a fusion system over a discrete $p$-toral group $S$.
\begin{itemize} 

\item A \emph{fusion subsystem} of $\calf$ is a 
subcategory $\cale$ of $\calf$ whose objects are the subgroups of some 
$T\le S$, and such that $\cale$ is itself a fusion system over $T$. 

\item For $T\le S$, $\calf|_{\le T}$ denotes the full subcategory of 
$\calf$ whose objects are the subgroups of $T$, regarded as a fusion 
subsystem of $\calf$ over $T$.

\end{itemize}
\end{DefNot}

The fusion subsystem $\calf|_{\le T}$ is not, in general, saturated, not 
even when $\calf$ is saturated. But in certain specialized cases this is 
the case.

\begin{Lem} \label{l:S=TCS(T)}
Let $\calf$ be a saturated fusion system over a discrete $p$-toral 
group $S$. Assume $T\le S$ is strongly closed in $\calf$, and is such that 
$S=TC_S(T)$. Then $\calf|_{\le T}$ is a saturated fusion subsystem of 
$\calf$.
\end{Lem}

\begin{proof} Set $\cale=\calf|_{\le T}$ for short: by definition, a 
fusion system over $T$. Fix $P\le T$, and choose $R\in P^\calf$ which 
is fully automized and receptive in $\calf$. Then $R\le T$ since $T$ is 
strongly closed, and $R\in P^\cale$ since $\cale$ is a full 
subcategory. Also, $\Aut_\cale(R)=\autf(R)$ (again since $\cale$ is a 
full subcategory), and $\Aut_T(R)=\Aut_S(R)$ since 
$N_S(R)=N_T(R)C_S(T)$. So $R$ is fully automized in $\cale$. 

If $\varphi\in\Iso_\cale(Q,R)$, then since $R$ is receptive in $\calf$, 
$\varphi$ extends to some $\4\varphi\in\homf(N_\varphi^\calf,S)$, where 
$N_\varphi^\calf\le N_S(Q)$ is as defined in Definition 
\reff{d:sat.f.s.}(c). Then $N_\varphi^\cale=N_\varphi^\calf\cap T$, and 
$\4\varphi(N_\varphi^\cale)\le T$ since $T$ is strongly closed. So 
$\4\varphi$ restricts to $\til\varphi\in\Hom_\calf(N_\varphi^\cale,T) 
=\Hom_\cale(N_\varphi^\cale,T)$, and 
since $\varphi$ was arbitrary, $R$ is receptive in $\cale$.

Thus each subgroup of $T$ is $\cale$-conjugate to one that is fully 
automized and receptive in $\cale$. The continuity axiom for 
$\cale$ follows immediately from that for $\calf$, and so $\cale$ is saturated. 
\end{proof}

We end the section with another technical lemma, one that will be needed in 
Section 4.

\begin{Lem} \label{l:Fcr/S1xS2}
Assume $S=S_1\times S_2$, where $S_1$ and $S_2$ are discrete $p$-toral groups, and 
let $\calf$ be a saturated fusion system over $S$. Then for each $P\le S$ 
that is $\calf$-centric and $\calf$-radical, there are subgroups $P_i\le 
S_i$ ($i=1,2$) such that $P=P_1\times P_2$.
\end{Lem}

\begin{proof} When $S$ is a finite $p$-group, this is shown in \cite[Lemma 
3.1]{AOV1}. We adapt that proof to fit this more general situation, while 
dealing with the extra complications that arise when the groups are 
infinite.

Let $\pr_i\:S\too S_i$ be the projection ($i=1,2$). For each $P\le S$, we 
write $P_i=\pr_i(P)$, and set $\5P=P_1\times P_2\ge P$. Let $\til{Z}_n(-)$ 
be as in Lemma \reff{l:Ln(S)}. We first claim that for each $n\ge1$, \beqq 
\til{Z}_n(P)=P\cap \til{Z}_n(\5P). \label{e:Ln(P)} \eeqq For $n=1$, this 
holds since $Z(P)=P\cap Z(\5P)$. If \eqreff{e:Ln(P)} holds for $n\ge1$, 
then we can identify $P/\til{Z}_n(P)$ as a subgroup of $\5P/\til{Z}_n(\5P)$ 
(which projects surjectively to each factor $P_i/\til{Z}_n(P_i)$), and 
\eqreff{e:Ln(P)} follows for $n+1$ since 
$\Omega_1(Z(P/\til{Z}_n(P)))=\Omega_1(Z(\5P/\til{Z}_n(\5P)))$.

Set $B=\{\alpha\in\autf(P)\,|\,[\alpha,\til{Z}_i(P)]\le 
\til{Z}_{i-1}(P)~\forall\,i\}$. Then $\Inn(P)\le B$, and $B\nsg\autf(P)$ 
since the $\til{Z}_i(P)$ are all characteristic. By Lemma \reff{l:Ln(S)}, 
the only element of order prime to $p$ in $B$ is the identity. If 
$\alpha\in B$ is such that its class $[\alpha]\in B/\Inn(P)$ has order $n$ 
prime to $p$, then $\alpha^n\in\Inn(P)$ has $p$-power order, so $\alpha$ 
has order $np^k$ for some $k\ge0$, and there is $\beta\in\gen\alpha$ of 
order $n$ such that $[\beta]=[\alpha]$. So the finite group 
$B/\Inn(P)\nsg\outf(P)$ has $p$-power order, and hence $B/\Inn(P)\le 
O_p(\outf(P))$. Thus $B=\Inn(P)$, since $P$ is $\calf$-radical. 

Assume that $P<\5P=P_1\times P_2$. Then $P<N_{\5P}(P)$ (see \cite[Lemma 
1.8]{BLO3}). Choose $x\in N_{\5P}(P)\sminus P$, and let $c_x\in\Aut_S(P)$ 
be conjugation by $x$. For each $n\ge1$, $c_x$ induces the identity on 
$\til{Z}_n(P)/\til{Z}_{n-1}(P)$ by \eqreff{e:Ln(P)} and since it induces 
the identity on $\til{Z}_n(\5P)/\til{Z}_{n-1}(\5P)$. Thus $c_x\in 
B=\Inn(P)$, and $x\in PC_S(P)=P$ since $P$ is $\calf$-centric, 
contradicting our assumption. We conclude that $P=\5P$ is a product, as 
claimed in the lemma.
\end{proof}

\nnewpage

\section{Morphisms of fusion systems and commuting fusion subsystems}

We are now ready to define morphisms of fusion systems.

\begin{Defi} \label{d:Mor(E,F)}
Let $\cale$ and $\calf$ be fusion systems over discrete $p$-toral groups 
$T$ and $S$, respectively. 
\begin{enuma}

\item A \emph{morphism} from $\cale$ to $\calf$ is a \new{pair $(f,\5f)$, 
where $f\in\Hom(T,S)$ and $\5f\:\cale\too\calf$ is a functor satisfying} 
\medskip
\begin{itemize} 
\item for each $P\le T$, $\5f(P)=f(P)\le S$; and 

\item for each $P,Q\le T$, each $\varphi\in\Hom_\cale(P,Q)$, and each $x\in 
P$, we have 
	\[ \5f(\varphi)(f(x))=f(\varphi(x))\in f(Q). \]


\end{itemize} 
We let $\Mor(\cale,\calf)\subseteq \Hom(T,S)$ 
denote the set of morphisms from $\cale$ to $\calf$, set 
$\End(\calf)=\Mor(\calf,\calf)$, and let $\Aut(\calf)\le\End(\calf)$ be the 
group of invertible endomorphisms. 

\item For each morphism \new{$(f,\5f)\in\Mor(\cale,\calf)$, define} \medskip
\begin{itemize} 

\item $\new{\Ker(f,\5f)} = \Ker\bigl(f\:T\too S\bigr) \nsg T$ (i.e., the kernel of 
$f$ as a group homomorphism); and 

\item $\new{\Im(f,\5f)} =\gen{\5f(\cale)}\le \calf$ (the smallest fusion subsystem of 
$\calf$ containing $\5f(\cale)$). 

\end{itemize} \medskip
By comparison, we write $f(T)\le S$ to denote the image of $f$ as a group 
homomorphism.
We say that \new{$(f,\5f)$} is \emph{surjective} (or onto) if 
$\Im(f,\5f)=\calf$; i.e., if each morphism in $\calf$ is a composite of 
morphisms in $\5f(\cale)$. 
\end{enuma}
\end{Defi}

\new{When $(f,\5f)\in\Mor(\cale,\calf)$, the conditions in Definition 
\ref{d:Mor(E,F)}(a) relating $f$ and $\5f$ make it clear that $\5f$ is 
uniquely determined by $f$. For this reason, when there is no risk of 
confusion, we usually just write $f\in\Mor(\cale,\calf)$ to represent the 
pair $(f,\5f)$.}

If $\cale$ is a saturated fusion system over a \emph{finite} $p$-group 
$T$ and $\new{(f,\5f)}\in\Mor(\cale,\calf)$, then by a theorem of Puig, 
$\new{\Im(f,\5f)}=\5f(\cale)$ and is a saturated fusion system. See, e.g., 
Corollary 5.15 and Proposition 5.11 in \cite{Craven} for details. But 
we do not know whether this is always the case when $T$ is an infinite 
discrete $p$-toral group, nor even whether $\5f(\cale)$ is always a 
subcategory. If one allows $\cale$ not to be saturated, then one can 
easily construct morphisms $\new{(f,\5f)}\in\Mor(\cale,\calf)$ where $\5f(\cale)$ 
is not a subcategory of $\calf$. 

However, it turns out that none of this is relevant when proving 
Theorem \reff{ThA}, which is why many of the statements in this section 
and the next involve fusion subsystems that need not be saturated, or 
morphisms whose domain is not assumed saturated. Later, in Proposition 
\reff{p:normal.end.2}(d), we will show that in the important case where 
$f$ is a \emph{normal} endomorphism of a saturated fusion system 
$\calf$, the image $\Im(f)$ is always saturated (and 
$\Im(f)=\5f(\calf)$).

\begin{Lem} \label{l:inj+surj=iso}
Let $\cale$ and $\calf$ be fusion systems over 
discrete $p$-toral groups $T$ and $S$.
\begin{enuma} 
\item For each $f\in\Mor(\cale,\calf)$, $\Ker(f)$ is strongly closed in 
$\cale$ (Definition \reff{d:Z+foc}).

\item If $f\in\Mor(\cale,\calf)$ is such that $\Ker(f)=1$ and 
$\Im(f)=\calf$, then $f$ is an isomorphism of fusion systems. 
\end{enuma}
\end{Lem}

\begin{proof} \textbf{(a) } For each $x\in\Ker(f)$ and $y\in x^\cale$, 
there is $\varphi\in\Mor(\cale)$ such that $y=\varphi(x)$, and hence 
$f(y)=\5f(\varphi)(f(x))=1$. So $x^\calf\subseteq\Ker(f)$, and 
$\Ker(f)$ is strongly closed in $\cale$. 

\smallskip

\noindent\textbf{(b) } If \new{$(f,\5f)\in\Mor(\cale,\calf)$ and} 
$\Ker(f)=1$, then the functor $\5f\:\cale\too\calf$ is injective on objects 
and on morphisms, and $\5f(\cale)\cong\cale$ is a fusion system. If in 
addition, $\Im(f)=\gen{\5f(\cale)}=\calf$, then $\5f(\cale)=\calf$, so 
$\5f$ is bijective, $\5f^{-1}$ is also a functor, and hence 
$\new{(f,\5f)^{-1}}\in\Mor(\calf,\cale)$. 
\end{proof}

We next recall the definition of a direct product of fusion systems. 

\begin{Defi} \label{d:F1xF2}
Let $\calf_1,\dots,\calf_k$ be fusion systems over discrete $p$-toral 
groups $S_1,\dots,S_k$. Set $S=\xxx{S}$, and let $\pr_i\:S\too S_i$ be 
projection to the $i$-th factor. The direct product $\calf=\xxx\calf$ 
is the fusion system over $S$ with morphism sets defined by 
	\[ \Hom_\calf(P,Q) = \bigl\{(\varphi_1,\dots,\varphi_k)|_P 
	\,\big|\, \varphi_i\in\Hom_{\calf_i}(\pr_i(P),\pr_i(Q)),~ 
	(\varphi_1,\dots,\varphi_k)(P)\le Q \bigr\}. \]
\end{Defi}

By construction, a product of fusion systems is again a fusion system. 
In fact, formally, $\calf=\xxx\calf$ is a product in the category 
of fusion systems and morphisms between them. For example, one easily 
checks that if $\calf^*$ is any fusion system over $S=\xxx{S}$ such that 
$\pr_i\in\Mor(\calf^*,\calf_i)$ for all $i$ (in the notation of Definition 
\reff{d:F1xF2}), then $\calf^*\le\calf$. 

If $S_i$ is finite and $\calf_i$ is saturated for all $i$, then the product 
$\calf$ is also saturated: see, e.g., \cite[Theorem I.6.6]{AKO}. That 
proof can easily be extended to show that products of saturated fusion systems 
over discrete $p$-toral groups are saturated, but since this will not be 
needed here, we omit it. Note that the converse follows from Lemma 
\reff{l:S=TCS(T)}: if a product of fusion systems is saturated, then so is 
each factor.

The following definition of commuting fusion subsystems is equivalent to 
that used by Henke \cite[Definition 3.1]{Henke} of ``subsystems that 
centralize each other''. The equivalence of the two definitions, at least 
in the finite case, is essentially the content of \cite[Proposition 
3.3]{Henke}. (See also the remarks after Lemma \reff{l:comm.subsyst.}.)

\begin{Defi} \label{d:comm.subsyst.}
Let $\calf$ be a fusion system over a discrete $p$-toral group 
$S$, and let $\cale_1,\dots,\cale_k\le\calf$ be fusion subsystems. We say 
that \emph{$\cale_1,\dots,\cale_k$ commute} if there is a morphism of 
fusion systems $\new{(I,\5I)}\in\Mor(\xxx\cale,\calf)$ whose restriction to each factor 
$\cale_i$ is the inclusion. In this situation, we set 
$\cale_1\cdots\cale_k=\new{\Im(I,\5I)=\5I(\xxx\cale)}$. 
\end{Defi}

The morphism $\new{(I,\5I)}\in\Mor(\xxx\cale,\calf)$ is uniquely determined 
whenever it exists. So $\cale_1\cdots\cale_k$ is well defined, and is the 
(unique) smallest fusion subsystem of $\calf$ in which the $\cale_i$ 
commute. By comparison, $\gen{\cale_1,\dots,\cale_k}$ is defined to be the 
smallest fusion subsystem of $\calf$ containing all of the $\cale_i$, and 
is in general smaller than $\cale_1\cdots\cale_k$.

This definition of commuting subsystems is, of course, motivated by one 
characterization of commuting subgroups of a group. But the following 
examples show that commuting subsystems can behave quite differently 
from commuting subgroups.

\begin{Ex} \label{ex:comm.or.not}
Set $p=3$. Fix groups $H_i\cong\Sigma_3$ for $i=1,2,3$, set 
$T_i=O_3(H_i)\cong C_3$, and choose $b_i\in H_i\sminus T_i$ (so 
$|b_i|=2$). Set 
	\[ \5G=H_1\times H_2\times H_3, \quad S=T_1\times T_2\times 
	T_3, \quad G=S\gen{b_1b_2,b_1b_3}<\5G, \]
and set $\5\calf=\calf_S(\5G)$ and $\calf=\calf_S(G)$. Also, set 
$\cale_i=\calf_{T_i}(H_i)$ (for $i=1,2,3$), so that 
$\cale_i\le\calf\le\5\calf$ are all saturated fusion subsystems.
\begin{enuma} 

\item The subsystems $\cale_1$, $\cale_2$, and $\cale_3$ commute pairwise 
in $\calf$, but do not commute as a triple.

\item The saturated fusion subsystems $\cale_1\cale_2$ and $\cale_3$ 
commute in $\5\calf$, but do not commute in $\calf<\5\calf$.

\end{enuma}
\end{Ex}

Of particular interest is the situation where $\cale_1$ and $\cale_2$ 
commute in $\calf$ and $\calf=\cale_1\cale_2$. 

\begin{Lem} \label{l:F=E1xE2}
Let $\calf$ be a saturated fusion system over a discrete $p$-toral group 
$S$, and let $\cale_1,\cale_2\le\calf$ be fusion subsystems over 
$T_1,T_2\le S$. Assume that $\cale_1$ and $\cale_2$ commute, and that 
$\cale_1\cale_2=\calf$ (thus $T_1T_2=S$). Then 
\begin{enuma} 
\item $T_1\cap T_2\le Z(\calf)$; and 
\item if $T_1\cap T_2=1$, then $\calf\cong\cale_1\times\cale_2$. 
\end{enuma}
\end{Lem}

\begin{proof} Let $\new{(I,\5I)}\in\Mor(\cale_1\times\cale_2,\calf)$ be the 
morphism that extends the inclusions. By assumption, 
$\calf=\cale_1\cale_2=\gen{\5I(\cale_1\times\cale_2)}$.

\smallskip

\noindent\textbf{(a) } Fix $x\in T_1\cap T_2$; we show that 
$x^\calf=\{x\}$. Since $\gen{\5I(\cale_1\times\cale_2)}=\calf$, it 
suffices to show that $\5I(\varphi_1,\varphi_2)(x)=x$ for each 
$\varphi_i\in\Hom_{\cale_i}(P_i,T_i)$ ($i=1,2$) such that 
$x\in P_1P_2$. Fix such $P_i$ and $\varphi_i$, and let 
$x_i\in P_i\le T_i$ be such that $x=x_1x_2$. Note that 
$x_1=xx_2^{-1}\in T_2$, and similarly $x_2\in T_1$.

By Definition \reff{d:Mor(E,F)}(a), 
$\5I(\varphi_1,\Id_{T_2})\in\homf(P_1T_2,S)$ sends $I(x_1,x_1^{-1})=1$ to 
$I(\varphi_1(x_1),x_1^{-1})=\varphi_1(x_1)x_1^{-1}$. Hence 
$\varphi_1(x_1)=x_1$. Also, $\varphi_2(x_2)=x_2$ by a similar argument, and 
so $\5I(\varphi_1,\varphi_2)(x)=\varphi_1(x_1)\varphi_2(x_2) =x_1x_2=x$. 
Hence $x^\calf=\{x\}$, and $x\in Z(\calf)$ by Lemma \reff{l:Z+foc}(a).

\smallskip

\noindent\textbf{(b) } If $T_1\cap T_2=1$, then $\Ker(I)=1$, and $I$ 
is an isomorphism of fusion systems by Lemma \reff{l:inj+surj=iso}(b) 
and since $\Im(I)=\cale_1\cale_2=\calf$.
\end{proof}

Motivated by Lemma \reff{l:F=E1xE2}, we now write 
$\calf=\cale_1\times\cale_2$ to mean that $\cale_1$ and $\cale_2$ are 
subsystems over $T_1$ and $T_2$ that commute in $\calf$, such that $T_1\cap 
T_2=1$ and $\cale_1\cale_2=\calf$. More generally, if 
$\cale_1,\dots,\cale_k$ is a $k$-tuple of commuting fusion subsystems of 
$\calf$, then we write $\calf=\xxx\cale$ to mean that the morphism 
$I\in\Mor(\xxx\cale,\calf)$ extending the inclusions is an isomorphism 
of fusion systems.

A fusion system $\calf$ over $S$ is \emph{indecomposable} if there are 
no fusion subsystems $\cale_1,\cale_2$ commuting in $\calf$, over 
proper subgroups $T_1,T_2<S$, such that $\calf=\cale_1\times\cale_2$. 
Our goal in the rest of the paper is to prove the essential uniqueness 
of factorizations of saturated fusion systems as products of 
indecomposable subsystems. The existence of such a factorization is 
elementary, based on the fact that discrete $p$-toral groups 
are artinian.

\begin{Prop} \label{p:exists.fact.}
Let $\calf$ be a fusion system over a discrete $p$-toral group $S$. 
Then there exist indecomposable fusion subsystems 
$\cale_1,\dots,\cale_k\le\calf$ such that $\calf=\xxx\cale$.
\end{Prop}

\begin{proof} If there is no such factorization, then there is a 
descending sequence of fusion subsystems 
$\calf=\cale_0\ge \cale_1\ge \cale_2\ge \cdots$ over subgroups 
$S=T_0\ge T_1\ge T_2\ge \cdots$, where $\cale_i$ is a proper direct factor of 
$\cale_{i-1}$ for each $i\ge1$ and hence $T_i<T_{i-1}$. But this is 
impossible, since $S$ is artinian (Lemma \reff{l:artinian}).
\end{proof}

In the next lemma, we give another, equivalent, condition for fusion 
subsystems to be commuting.

\begin{Lem} \label{l:comm.subsyst.}
Let $\calf$ be a fusion system over a discrete $p$-toral group 
$S$, and let $\cale_1,\dots,\cale_k\le\calf$ be fusion subsystems over 
subgroups $T_1,\dots,T_k\le S$. Then the following are equivalent.
\begin{enuma} 

\item The subsystems $\cale_1,\dots,\cale_k$ commute.

\item The subgroups $T_i$ commute pairwise, and for each 
$k$-tuple of morphisms 
	\[ \bigl\{\varphi_i\in\Hom_{\cale_i}(P_i,Q_i)\bigr\}_{i=1}^k \in 
	\Mor(\cale_1)\times\cdots\times\Mor(\cale_k), \]
there is $\4\varphi\in\homf(P_1\cdots P_k,Q_1\cdots Q_k)$ 
such that $\4\varphi|_{P_i}=\varphi_i$ for each $i$.

\end{enuma}
\end{Lem}

\begin{proof} \noindent\boldd{(a$\implies$b) } Let 
$\new{(I,\5I)}\in\Mor(\cale_1\times\cdots\times\cale_k,\calf)$ be the morphism that 
extends the inclusions. Thus $I(x_1,\dots,x_k)=x_1\cdots x_k$ for $x_i\in 
T_i$, so the subgroups $T_i$ commute pairwise. If 
$\{\varphi_i\in\Hom_{\cale_i}(P_i,Q_i)\}$ is a $k$-tuple of morphisms and 
$\4\varphi=\5I(\varphi_1,\dots,\varphi_k)$, then 
for each $(x_1,\dots,x_k)$ with $x_i\in P_i$, 
	\[ \4\varphi(x_1\cdots x_k) = 
	\5I(\varphi_1,\dots,\varphi_k)(I(x_1,\dots,x_k)) = 
	I(\varphi_1(x_1),\dots,\varphi_k(x_k)) = 
	\varphi_1(x_1)\cdots\varphi_k(x_k). \]
So $\4\varphi\in\homf(P_1\cdots P_k,Q_1\cdots Q_k)$ extends each of the 
$\varphi_i$. 

\smallskip

\noindent\boldd{(b$\implies$a) } Since the $T_i$ commute pairwise, we can 
define $I\in\Hom(\xxx{T},S)$ by setting $(x_1,\dots,x_k)=x_1\cdots x_k$ for 
$x_i\in T_i$. Define a functor $\5I\:\xxx\cale\too\calf$ as follows. On 
objects, we set $\5I(P)=I(P)$ for each $P\le\xxx{T}$ as usual. In 
particular, $\5I(\xxx{P})=P_1\cdots P_k$ for $P_i\le T_i$. 

By definition of the product fusion system, for each 
$\varphi\in\Hom_{\cale_1\times\cdots\times\cale_k}(P,Q)$, there are 
morphisms $\varphi_i\in\Hom_{\cale_i}(P_i,Q_i)$ (for $1\le i\le k$) 
such that 
	\[ P\le\xxx{P}, \quad Q\le\xxx{Q}, \quad\textup{and}\quad 
	\varphi=(\xxx\varphi)|_P. \]
By assumption (b), there is a morphism $\4\varphi\in\homf(P_1\cdots 
P_k,Q_1\cdots Q_k)$ that extends each of the $\varphi_i$, this is clearly 
unique, and we set $\5I(\xxx\varphi)=\4\varphi$ and 
$\5I(\varphi)=\4\varphi|_{I(P)}$. By the uniqueness of $\4\varphi$, these 
preserve composition, and hence define a functor $\5I$ from $\xxx\cale$ to 
$\calf$ associated to $I$. Thus \new{$(I,\5I)$} is a morphism of fusion 
systems, and the subsystems $\cale_i$ commute.
\end{proof}

In \cite[Definition 3.1]{Henke}, Henke defined two fusion subsystems 
$\cale_1,\cale_2\le\calf$ over $T_1,T_2\le S$ to commute in $\calf$ if 
$\cale_1\le C_\calf(T_2)$ and $\cale_2\le C_\calf(T_1)$. (See, e.g., 
\cite[Definition I.5.3]{AKO} for the definition of centralizer fusion 
systems.) It is not hard to see that this is equivalent to condition (b) in 
Lemma \reff{l:comm.subsyst.}.

The next lemma describes some of the elementary relations among 
commuting subsystems, including a form of associativity.

\begin{Lem} \label{l:Z(F)xF}
Let $\calf$ be a fusion system over a discrete $p$-toral 
group $S$, and let $\cale_1,\dots,\cale_k\le\calf$ be fusion subsystems 
over $T_1,\dots,T_k\le S$. 
\begin{enuma} 

\item For $2\le\ell<k$, $\cale_1,\dots,\cale_k$ commute in $\calf$ if and 
only if there is $\calf^*\le\calf$ such that $\cale_1,\dots,\cale_\ell$ 
commute in $\calf^*$ and $\calf^*,\cale_{\ell+1},\dots,\cale_k$ commute in 
$\calf$.

\item Assume $f\in\Mor(\calf,\cald)$ for some fusion system $\cald$. If 
$\cale_1,\dots,\cale_k$ commute in $\calf$, then the subsystems 
$\Im(f|_{\cale_1}),\dots,\Im(f|_{\cale_k})$ commute in $\Im(f)$.

\end{enuma}
\end{Lem}

\begin{proof} \noindent\textbf{(b) } Let 
$\{\varphi_i\in\Hom_{\cale_i}(P_i,Q_i)\}_{i=1}^k$ be a $k$-tuple of 
morphisms. By Lemma \reff{l:comm.subsyst.}(a$\Rightarrow$b) and since the 
$\cale_i$ commute in $\calf$, there is $\4\varphi\in\Hom_\calf(P_1\cdots 
P_k,Q_1\cdots Q_k)$ that extends each of the $\varphi_i$. Then 
$\5f(\4\varphi)\in\Hom_{\Im(f)}(f(P_1\cdots P_k),f(Q_1\cdots Q_k))$ extends 
each of the $\5f(\varphi_i)$.

Thus each $k$-tuple in $\prod_{i=1}^k\Mor(\5f(\cale_i))$ extends to a 
morphism in $\Im(f)$. The same is true for a $k$-tuple in 
$\prod_{i=1}^k\Mor(\Im(f|_{\cale_i}))$ since each morphism in 
$\Im(f|_{\cale_i})=\gen{\5f(\cale_i)}$ is a composite of morphisms in 
$\5f(\cale_i)$, and so the subsystems 
$\Im(f|_{\cale_1}),\dots,\Im(f|_{\cale_k})$ commute in $\Im(f)$ by 
Lemma \reff{l:comm.subsyst.}(b$\Rightarrow$a). 

\smallskip

\noindent\textbf{(a) } Assume first that there is $\calf^*\le\calf$ such 
that $\cale_1,\dots,\cale_\ell$ commute in $\calf^*$ and also 
$\calf^*,\cale_{\ell+1},\dots,\cale_k$ commute in $\calf$. Thus there are 
morphisms 
	\[ J_1\in\Mor(\calf^*\times\cale_{\ell+1}\times\cdots\times\cale_k,\calf) 
	\qquad\textup{and}\qquad
	J_2\in\Mor(\cale_1\times\cdots\times\cale_\ell,\calf^*) \]
each of which extends the inclusions of the different factors into the 
target. Then 
$J_1\circ(J_2\times\Id_{\cale_{\ell+1}\times\cdots\times\cale_k})$ is 
a morphism of fusion systems from $\xxx\cale$ to 
$\calf$ extending the inclusions of the $\cale_i$, so these 
subsystems commute. 

Conversely, assume $\cale_1,\dots,\cale_k$ commute in $\calf$, let 
$I\in\Mor(\xxx\cale,\calf)$ extend the inclusions, and set 
$\calf^*=\cale_1\cdots\cale_\ell$. The subsystems 
$(\xxx[\ell]\cale),\cale_{\ell+1},\dots,\cale_k$ commute in $\xxx\cale$. So 
by (b), applied with $I$ in the role of $f$, 
$(\cale_1\cdots\cale_\ell),\cale_{\ell+1},\dots,\cale_k$ commute in $\Im(I)$ 
and hence in $\calf$. 
\end{proof}

The next lemma gives conditions for a saturated fusion system to factorize 
as a product when the underlying discrete $p$-toral group factorizes.

\begin{Lem} \label{l:F=ExD}
Let $\calf$ be a saturated fusion system over a discrete $p$-toral group 
$S$. Assume $T,U\le S$ are such that $S=T\times U$, and set 
$\cale=\calf|_{\le T}$ and $\cald=\calf|_{\le U}$. Assume also that $T$ and 
$U$ are strongly closed in $\calf$, and that $\cale$ and $\cald$ commute in 
$\calf$. Then $\calf=\cale\times\cald$.
\end{Lem}

\begin{proof} Let $I\in\Mor(\cale\times\cald,\calf)$ be the morphism 
that extends the inclusions. Then $\Ker(I)=T\cap U=1$, so by Lemma 
\reff{l:inj+surj=iso}, $\Im(I)\le\calf$ is a fusion subsystem over $S=TU$. 
We will show that $\Im(I)=\calf$. Once this has been shown, then 
$\calf=\cale\times\cald$ by Lemma \reff{l:F=E1xE2}(b). 

Assume $P\le S$ is $\calf$-centric and $\calf$-radical (Definition 
\ref{d:Z+foc}). By Lemma 
\reff{l:Fcr/S1xS2}, we have $P=P_1\times P_2$, where $P_1=P\cap T$ and 
$P_2=P\cap U$. For each $\alpha\in\autf(P)$, 
$\alpha|_{P_i}\in\autf(P_i)$ for $i=1,2$ since $T$ and $U$ are strongly 
closed, and hence $\alpha|_{P_1}\in\Aut_\cale(P_1)$ and 
$\alpha|_{P_2}\in\Aut_\cald(P_2)$ since $\cale$ and $\cald$ are full 
subcategories. So $\alpha=\5I(\alpha|_{P_1},\alpha|_{P_2})\in\Aut_{\Im(I)}(P)$. 

By the version of Alperin's fusion theorem shown in \cite[Theorem 
3.6]{BLO3}, each morphism in $\calf$ is a (finite) composite of 
restrictions of $\calf$-automorphisms of $\calf$-centric $\calf$-radical 
subgroups of $S$. So every morphism in $\calf$ is in $\Im(I)$, and 
$\calf=\Im(I)=\cale\times\cald$. 
\end{proof}

\nnewpage

\section{Sums and summability of endomorphisms}

\newcommand\zero[1]{{\bf0}_{#1}}

We next define sums of morphisms between a pair of fusion systems.

\begin{Defi} \label{d:f1+f2}
Let $\cale$ and $\calf$ be fusion systems over discrete $p$-toral 
groups $T$ and $S$. 
\begin{enuma} 

\item A $k$-tuple of morphisms $f_1,\dots,f_k\in\Mor(\cale,\calf)$ 
(for $k\ge2$) is 
\emph{summable} if the fusion subsystems $\Im(f_1),\dots,\Im(f_k)$ commute. 
When this is the case, $f_1+\dots+f_k\in\Hom(T,S)$ is the 
morphism 
	\[ (f_1+\dots+f_k)(x) = f_1(x)\cdots f_k(x)\in S 
	\qquad\qquad (\textup{all}~ x\in T). \]

\item Let $\zero{}=\zero{\cale,\calf}\in\Mor(\cale,\calf)$ be the neutral 
element for sums of morphisms: the homomorphism sending $T$ to 
$1\in S$. Write $\zero\calf=\zero{\calf,\calf}\in\End(\calf)$ for short.

\end{enuma}
\end{Defi}

We first check that a sum of summable morphisms from $\cale$ to $\calf$ is, 
in fact, a morphism from $\cale$ to $\calf$.

\begin{Lem} \label{l:f1+f2}
Let $\cale$ and $\calf$ be fusion systems over discrete $p$-toral 
groups $T$ and $S$, and let 
$f_1,\dots,f_k\in\Mor(\cale,\calf)$ be a summable $k$-tuple of morphisms. 
Then $f_1+\dots+f_k\in\Mor(\cale,\calf)$, and 
$\Im(f_1+\dots+f_k)\subseteq\Im(f_1)\cdots\Im(f_k)$. 
\end{Lem}

\begin{proof} Set $f=f_1+\dots+f_k$ for short. As a homomorphism of groups, 
$f$ is the composite
	\[ T \Right6{(f_1,\dots,f_k)} f_1(T)\times\cdots\times f_k(T) 
	\Right4{I} S, \]
where $I(x_1,\dots,x_k)=x_1\cdots x_k$ for $x_i\in f_i(T)\le S$. By 
assumption, there are functors $\5f_i$ associated to $f_i$ and $\5I$ 
associated to $I$, and we let $\5f^*$ denote the composite functor 
	\[ \5f^*\: \cale \Right6{(\5f_1,\dots,\5f_k)} \Im(f_1) 
	\times\cdots\times \Im(f_k) \Right4{\5I} \calf. \]
Then $\5f^*(P)=f_1(P)\cdots f_k(P)$ for $P\le T$, and the following diagram 
of groups commutes for each $\varphi\in\Hom_\cale(P,Q)$: 
	\[ \xymatrix@C=50pt@R=25pt{ 
	P \ar[r]^-{f|_P} \ar[d]^{\varphi} & 
	f_1(P)\cdots f_k(P) \ar[d]^{\5f^*(\varphi)} \\
	Q \ar[r]^-{f|_Q} & f_1(Q)\cdots f_k(Q).
	} \] 
Thus $\5f^*(\varphi)(f(P))\le f(Q)$. So if we define $\5f\:\cale\too\calf$ 
to be the functor that sends $P$ to $f(P)$ and sends 
$\varphi\in\Hom_\cale(P,Q)$ to the restriction $\5f^*(\varphi)|_{f(P)}$, 
then this is a functor associated to $f$. So $f\in\Mor(\cale,\calf)$.

In particular, every morphism in $\Im(f)$ is the restriction of a morphism 
in $\Im(I)$, and hence is also in $\Im(I)$. Since 
$\Im(f_1)\cdots\Im(f_k)=\Im(I)$ by definition, this proves that 
$\Im(f)\le\Im(f_1)\cdots\Im(f_k)$.
\end{proof}

We next check that a composite of sums of morphisms is a sum of composites 
in the way one expects. Since this is clear on the level of group 
homomorphisms, the main problem is to check summability.

\begin{Lem} \label{l:(f1+f2)(f3+f4)}
Let $\cald$, $\cale$, and $\calf$ be fusion systems over discrete 
$p$-toral groups, and assume that $f_1,\dots,f_n\in\Mor(\cale,\calf)$ and 
$g_1,\dots,g_m\in\Mor(\cald,\cale)$ are summable $m$- and $n$-tuples of 
morphisms (some $m,n\ge1$). Then $\{f_ig_j \,|\, 1\le i\le n,~ 1\le j\le 
m\}$ is summable, and 
	\beqq (f_1+\dots+f_n)\circ(g_1+\dots+g_m) = 
	\sum_{i=1}^n\sum_{j=1}^m f_ig_j. \label{e:distrib} \eeqq
\end{Lem}

\begin{proof} By assumption, the fusion subsystems 
$\Im(g_1),\dots,\Im(g_m)$ commute in $\cale$. So for each $1\le i\le 
n$, the subsystems $\Im(f_ig_1),\dots,\Im(f_ig_m)$ commute in $\Im(f_i)$ 
by Lemma \reff{l:Z(F)xF}(b). Since the $\Im(f_i)$ commute in $\calf$, 
the subsystems $\Im(f_ig_j)$ for all $1\le i\le n$ and $1\le j\le m$ commute 
in $\calf$ by repeated applications of Lemma \reff{l:Z(F)xF}(a). 

Thus the $f_ig_j$ are summable. Equation \eqreff{e:distrib} holds since 
for each $x\in U$, 
	\begin{align*} 
	(f_1+\dots+f_n)(g_1+\dots g_m)(x) 
	&= f_1(g_1(x)\cdots g_m(x)) \cdots f_n(g_1(x)\cdots g_m(x)) \\
	&= \textstyle\prod\limits_{i=1}^n\prod\limits_{j=1}^m f_ig_j(x) 
	= \biggl(\textstyle\sum\limits_{i=1}^n\sum\limits_{j=1}^m 
	f_ig_j\biggr)(x). \qedhere 
	\end{align*}
\end{proof}

\nnewpage

\section{Normal endomorphisms}

An endomorphism of a group $G$ is defined to be normal if it commutes with 
all inner automorphisms of $G$, \new{and} it is not \new{at all} obvious 
how to translate this directly to an \new{analogous} definition for fusion 
systems. \new{We refer to Remark \ref{r:EndN(G)} below for more discussion 
about the difficulties with such a definition.} 

Instead, we use a different property of normal endomorphisms of 
groups. It is an easy exercise to show that $f\in\End(G)$ is normal if and 
only if there is $\chi\in\End(G)$ such that $[\Im(f),\Im(\chi)]=1$ and 
$f+\chi=\Id_G$, and this criterion is easily adapted to endomorphisms of 
fusion systems.

\begin{Defi} \label{d:normal.end.}
Let $\calf$ be a fusion system over a discrete $p$-toral group $S$. An 
endomorphism $f\in\End(\calf)$ is \emph{normal} if there is 
$\chi\in\End(\calf)$ such that $f$ and $\chi$ are summable and 
$f+\chi=\Id_\calf$. Let $\End^N(\calf)\ge\Aut^N(\calf)$ denote the sets 
of normal endomorphisms and automorphisms of $\calf$.
\end{Defi}

\new{As one example, assume $G$ is a finite group with $S\in\sylp{G}$, and 
let $f\in\End(G)$ be a normal endomorphism such that $f(S)\le S$. Then 
there is $\chi\in\End(G)$ such that $f+\chi=\Id_G$, and the corresponding 
relation holds for $f|_S$ and $\chi|_S$ as endomorphisms of the fusion 
system $\calf_S(G)$. So $f|_S$ is a normal endomorphism of $\calf_S(G)$.}

We first check some of the most basic properties of normal endomorphisms.

\begin{Lem} \label{l:normal.end.1}
Let $\calf$ be a saturated fusion system over a discrete $p$-toral group $S$. 
\begin{enuma}

\item If $f,f'\in\End^N(\calf)$, then $f\circ f'$ is normal. 

\item If $f,f'\in\End^N(\calf)$ and $f\circ f'=\zero\calf$, then $f$ and $f'$ 
are summable and $f+f'$ is normal.

\item If $\calf=\cale_1\times\cale_2$, and $f\in\End(\calf)$ is the 
identity on $\cale_1$ and trivial on $\cale_2$, then $f$ is normal.

\end{enuma}
\end{Lem}

\begin{proof} \textbf{(a,b) } Fix a pair $f,f'$ of normal endomorphisms of 
$\calf$, and let $\chi,\chi'\in\End(\calf)$ be such that $f$ and $\chi$ are 
summable, $f'$ and $\chi'$ are summable, and $f+\chi=\Id_\calf=f'+\chi'$. 
By Lemma \reff{l:(f1+f2)(f3+f4)}, $\Id_\calf = 
(f+\chi)\circ(f'+\chi')=ff'+(f\chi'+\chi f'+\chi\chi')$, and so $ff'$ is 
normal.

If $ff'=\zero\calf$, then by Lemma \reff{l:(f1+f2)(f3+f4)}, 
	\[ \Id_\calf = (f+\chi)\circ(f'+\chi') = f(f'+\chi') + (f+\chi)f' + 
	\chi\chi' = f + f' + \chi\chi'. \]
Hence $f$ and $f'$ are summable, and $f+f'$ is normal.

\smallskip

\noindent\textbf{(c) } If $\calf=\cale_1\times\cale_2$, and 
$f_1,f_2\in\End(\calf)$ are the projections with images $\cale_1$ and 
$\cale_2$, respectively, then $f_1+f_2=\Id_\calf$. So $f_1$ and $f_2$ are 
normal.
\end{proof}

We do not know whether or not the sum of each summable pair of normal 
endomorphisms is normal. (This is clearly true for normal endomorphisms of 
a group.) One can at least weaken the extra hypothesis in Lemma 
\reff{l:normal.end.1}(b): it suffices to assume that $ff'(S)\le Z(\calf)$. 

By point (a) above, $\End^N(\calf)$ is a monoid for each fusion system 
$\calf$. In the next lemma, we show that $\Aut^N(\calf)$ is always a group.

\begin{Lem} \label{l:[f,S]<Z(F)}
Let $\calf$ be a saturated fusion system over a discrete $p$-toral group 
$S$, and assume  $f\in\End(\calf)$ is surjective. Then 
\begin{enuma} 
\item $f$ is normal if and only if $[f,S]\le Z(\calf)$ and 
$f|_{\foc(\calf)}=\Id$; and 
\item if $f$ is normal and invertible, then $f^{-1}$ is also normal.
\end{enuma}
\end{Lem}

\begin{proof}
If $f$ is normal, then there is $\chi\in\End(\calf)$ such that $f$ and 
$\chi$ are summable and $f+\chi=\Id_\calf$, and in particular, 
$\Im(f)=\calf$ commutes with $\Im(\psi)$. So 
$[f,S]=\chi(S)=f(S)\cap\chi(S)\le Z(\calf)$ by Lemma \reff{l:F=E1xE2}(a). 
Hence $x\in S$ and $y\in x^\calf$ imply that $\chi(y)=\chi(x)$, so 
$\foc(\calf)\le\Ker(\chi)$, and $f|_{\foc(\calf)}=\Id$. 

Conversely, if $[f,S]\le Z(\calf)$, then we can define $\chi$ by 
setting $\chi(x)=f(x)^{-1}x\in Z(\calf)$ for $x\in S$. Then $\chi\in\End(S)$, 
$\chi(S)\le Z(\calf)$, and $f+\chi=\Id_S$ as endomorphisms of $S$. 
If in addition, $f|_{\foc(\calf)}=\Id$, then $\Ker(\chi)\ge\foc(\calf)$, so 
$\chi$ is constant on $\calf$-conjugacy classes. Hence there is a functor 
$\5\chi\:\calf\too\calf$ associated to $\chi$, defined on objects by 
setting $\5\chi(P)=\chi(P)\le Z(\calf)$ for each $P\le S$, and on morphisms 
by setting $\5\chi(\varphi)=\Id_{\chi(P)}$ for each $\varphi\in\homf(P,Q)$. 
So $\chi\in\End(\calf)$, and $f$ is normal. 

\smallskip

\noindent\textbf{(b) } If $f\in\End(\calf)$ is normal and invertible, then 
$[f^{-1},S]=[f,S]\le Z(\calf)$, and $f^{-1}|_{\foc(\calf)}=\Id$ since 
$f|_{\foc(\calf)}=\Id$. So $f^{-1}$ is normal by (a). 
\end{proof}

\new{By comparison, an automorphism of a group $G$ is normal if and only if 
it induces the identity on $[G,G]$ and on $G/Z(G)$ (see 
\cite[6.18.ii]{Sz1}). This gives another way to see that when $G$ is finite 
and $S\in\sylp{G}$, each normal automorphism of $G$ that sends $S$ to 
itself induces a normal automorphism of the fusion system $\calf_S(G)$.}

We can now show that the image of a normal endomorphism is always 
saturated and is a full subcategory.

\begin{Prop} \label{p:normal.end.2}
Let $\calf$ be a saturated fusion system over a discrete $p$-toral group 
$S$, let $\calf\in\End^N(\calf)$ be a normal endomorphism, and set $T=f(S)$. 
Let $\chi\in\End^N(\calf)$ be such that 
$f+\chi=\Id_\calf$, and set $U=\chi(S)$. Then 
\begin{enuma} 

\item $f\chi=\chi f$ and $f(U)=\chi(T)=T\cap U\le Z(\calf)$; 

\item $T$ is strongly closed in $\calf$ (see Definition \reff{d:Z+foc}); 

\item $f$ commutes in $\End(\calf)$ with all elements of 
$\autf(S)$; and 

\item $\Im(f)=\5f(\calf)=\calf|_{\le T}$ and is a saturated fusion 
subsystem of $\calf$.

\end{enuma}
\end{Prop}

\begin{proof} Since $f$ and $\chi$ are summable, we have $[T,U]=1$. Since 
$x=f(x)\chi(x)$ for each $x\in S$, we have $TU=S$. As usual, 
$\5f,\5\chi\:\calf\too\calf$ denote the functors associated to $f$ and 
$\chi$. 

\smallskip

\noindent\textbf{(a) } For each $x\in S$, \new{
	\[ 
	f(f(x))\chi(f(x)) = 
	f(x)=f(f(x)\chi(x))= f(f(x))f(\chi(x)). \] 
So $\chi f(x)=f\chi(x)$ for each $x\in S$, and $\chi f=f\chi\in\End(S)$.} 

In particular, 
$f(U)=f\chi(S)=\chi f(S)=\chi(T)\le T\cap U$. For each $x\in T\cap 
U$, $x=f(x)\chi(x)\in f(U)\chi(T)=f(U)$, and so $f(U)=\chi(T)=T\cap U$. 

Since $f$ and $\chi$ are summable and $f+\chi=\Id_\calf$, the subsystems 
$\Im(f)$ and $\Im(\chi)$ commute, and 
$\Im(f)\Im(\chi)=\calf$ by Lemma \reff{l:f1+f2}. Hence $T\cap 
U\le Z(\calf)$ by Lemma \reff{l:F=E1xE2}(a).

\smallskip

\noindent\textbf{(b) } Assume $t\in T$ and $\varphi\in\homf(\gen{t},S)$. 
Then $\chi(t)\in\chi(T)=T\cap U$. Also, 
$\varphi(t)=f(\varphi(t))\chi(\varphi(t))$, where $f(\varphi(t))\in 
f(S)=T$, and where $\chi(\varphi(t))=\5\chi(\varphi)(\chi(t))=\chi(t)\in 
T\cap U$ by (a) and since $\chi(t)\in T\cap U\le Z(\calf)$. Thus 
$\varphi(t)\in T$, so $t^\calf\subseteq T$, and $T$ is strongly closed in 
$\calf$. 


\smallskip

\noindent\textbf{(c) } Fix $\alpha\in\autf(S)$; we must show that 
$f\alpha=\alpha f$. For each $x\in S$, 
	\[ f(\alpha(x))\chi(\alpha(x)) = \alpha(x) = \alpha(f(x)\chi(x)) = 
	\alpha f(x)\alpha\chi(x), \]
so $f\alpha(x)=\alpha f(x)$ if and only if $\chi\alpha(x)=\alpha\chi(x)$.

For each $t\in T$, $\chi(t)\in\chi(T)\le Z(\calf)$ by (a), so 
	\[ \chi(\alpha(t)) = \5\chi(\alpha)(\chi(t)) = \chi(t) = 
	\alpha(\chi(t)). \]
Thus $\chi\alpha|_T=\alpha\chi|_T$, and by the above remarks, this implies 
$f\alpha|_T=\alpha f|_T$. The restrictions to $U$ commute by a similar 
argument. Since $S=TU$, we have $f\alpha=\alpha f$ and 
$\chi\alpha=\alpha\chi$. 

\smallskip

\noindent\textbf{(d) } Fix $P,Q\le T$ and $\varphi\in\homf(P,Q)$. We must 
show that $\varphi\in\Hom_{\5f(\calf)}(P,Q)$. 

If $\varphi(P)$ is receptive in $\calf$, then since $U\le C_S(P)\le 
N_\varphi^\calf$, $\varphi$ extends to some $\4\varphi\in\homf(PU,QU)$, where 
$\4\varphi(U)=U$ since $U$ is strongly closed by (b). If $\varphi(P)$ 
is not receptive, let $R\in P^\calf$ be such that $R$ is 
receptive; then $R\le T$ since $T$ is strongly closed, and 
there are isomorphisms $\4\varphi_1\in\homf(PU,RU)$ and 
$\4\varphi_2\in\homf(\varphi(P)U,RU)$ such that $\4\varphi_i(U)=U$, 
$\4\varphi_1(P)=R$, $\4\varphi_2(\varphi(P))=R$, and 
$\4\varphi_2^{-1}\4\varphi_1\in\homf(PU,\varphi(P)U)$ extends 
$\varphi$. So in all cases, we get $\4\varphi\in\homf(PU,QU)$ extending 
$\varphi$.

Set $\psi=\5f(\4\varphi)\in\Hom_{\5f(\calf)}(f(PU),f(QU))$. We claim that 
$\varphi$ is a restriction of $\psi$ (in particular, that $P\le f(PU)$). 
For each $x\in P$, $x=f(x)\chi(x)$ where $f(x)\in f(P)$ and 
$\chi(x)\in\chi(T)=f(U)\le Z(\calf)$, so $x\in f(PU)$ and 
	\begin{align*} 
	\psi(x) &= \psi(f(x)\chi(x)) = \5f(\4\varphi)(f(x))\cdot\psi(\chi(x)) \\
	&= f(\4\varphi(x))\cdot\chi(x) 
	= f(\varphi(x))\cdot\5\chi(\varphi)(\chi(x)) 
	\tag{$\chi(x)\in Z(\calf)$} \\
	&= f(\varphi(x))\chi(\varphi(x)) = \varphi(x) .
	\end{align*}
So $\psi|_P=\varphi$, and hence $\varphi\in\Hom_{\5f(\calf)}(P,Q)$. 

Thus $\Im(f)=\5f(\calf)=\calf|_{\le T}$: the full subcategory of $\calf$ 
with objects the subgroups of $T$. Since $S=TU$, $[T,U]=1$, and $T$ is 
strongly closed in $\calf$, Lemma \reff{l:S=TCS(T)} now implies that 
$\Im(f)$ is a saturated fusion subsystem of $\calf$. 
\end{proof}

\begin{New} 
\begin{Rmk} \label{r:EndN(G)}
As stated above, an endomorphism of a group is defined to be normal if it 
commutes with all inner automorphisms. When $\calf$ is a saturated fusion 
system over $S$, the closest we can come to inner automorphisms are the 
elements of $\autf(S)$, and Proposition \reff{p:normal.end.2}(c) says that 
each normal endomorphism of $\calf$ commutes with all of these. However, 
the converse is not true. If, for example, $S$ and $\autf(S)$ are both 
abelian, then each $\alpha\in\autf(S)$ lies in $\Aut(\calf)$ and commutes 
with $\autf(S)$. But $\alpha$ need not be the identity on 
$[S,\autf(S)]\le\foc(\calf)$, and hence by Lemma \reff{l:[f,S]<Z(F)}(a) need 
not be normal. 
\end{Rmk}
\end{New}

Recall (Definition \reff{d:loc.nilp.}) that when $\calf$ is a fusion system 
over $S$, an endomorphism $f\in\End(\calf)\le\End(S)$ is locally nilpotent 
if $S=\bigcup_{i=1}^\infty\Ker(f^i)$. If $S$ is finite, then clearly all 
locally nilpotent endomorphisms are nilpotent. The next lemma, which is 
modeled on (2.4.9) in \cite{Sz1}, implies among other things that this is 
also true whenever $Z(\calf)$ is finite.

\begin{Prop} \label{p:F=ExD}
Let $\calf$ be a saturated fusion system over a discrete $p$-toral group 
$S$, and let $f\in\End^N(\calf)$ be a normal endomorphism. Then there is a 
unique pair of saturated fusion subsystems $\cale,\cald\le\calf$, over 
$T,U\le S$, such that 
\begin{itemize} 
\item $\calf=\cale\times\cald$;
\item $f|_T\in\Aut^N(\cale)$; and 
\item $f|_U\in\End^N(\cald)$ is locally nilpotent, and is nilpotent if 
$Z(\calf)$ is finite. 
\end{itemize}
If $f$ is surjective, then \new{$U$ is connected and central in $\calf$,}
and hence $U=1$ if $|Z(\calf)|<\infty$. 
\end{Prop}

\begin{proof} Since $S$ is artinian (Lemma \reff{l:artinian}) and 
$\{f^i(S)\}_{i=1}^\infty$ is a descending sequence of subgroups of $S$, 
there is $n\ge1$ such that $f^i(S)=f^n(S)$ for all $i\ge n$. 
Consider the following subgroups of $S$ and fusion subsystems of $\calf$:
	\[ T_+=\bigcap\nolimits_{i=1}^\infty f^i(S)=f^n(S),\qquad 
	\cale_+=\calf|_{\le T_+}, \qquad 
	U=\bigcup\nolimits_{i=1}^\infty\Ker(f^i), \qquad \cald=\calf|_{\le 
	U}. \]

We first check that 
	\beqq S=T_+\Ker(f^n)=T_+U;
	\quad \textup{$T_+$, $U$ strongly closed in $\calf$;} 
	\quad \textup{$\cale_+$, $\cald$ commute in $\calf$.} 
	\label{e:tu=s} \eeqq
For each $x\in S$, since $f^{2n}(S)=f^n(S)$, there is $y\in T_+$ such that 
$f^n(y)=f^n(x)$ and $x=y(y^{-1}x)\in T_+\Ker(f^n)$. Thus 
$S=T_+\Ker(f^n)=T_+U$. Also, $T_+$ and $U$ are strongly closed in $\calf$ 
by Proposition \reff{p:normal.end.2}(b) and Lemma \reff{l:inj+surj=iso}(a), 
respectively. 

Let $\psi_n\in\End(\calf)$ be such that $f^n+\psi_n=\Id_\calf$. Then 
$\Im(f^n)$ and $\Im(\psi_n)$ commute in $\calf$, where $\Im(f^n)=\cale_+$ 
and $\Im(\psi_n)=\calf|_{\le\psi_n(S)}$ by Proposition 
\reff{p:normal.end.2}(d). Also, $U\le\psi_n(S)$, since for 
$u\in U$, $f^{(m+1)n}(u)=1$ for some $m\ge1$, and hence 
	\[ u = (uf^n(u)^{-1})(f^n(u)f^{2n}(u)^{-1})\cdots f^{mn}(u) = 
	\psi_n(uf^n(u) \cdots f^{mn}(u))
	\in \psi_n(S). \]
Thus $\cald\le\calf|_{\psi_n(S)}=\Im(\psi_n)$ commutes with $\cale_+$, 
finishing the proof of \eqreff{e:tu=s}.

\noindent\boldd{If $S$ is finite,} then set $T=T_+=f^n(S)$ and 
$\cale=\cale_+$. Since $|S|=|f^i(S)|\cdot|\Ker(f^i)|$ for all $i$, we have 
$U=\Ker(f^n)=\Ker(f^i)$ for all $i\ge n$, and $|S|=|T|\cdot|U|$. By 
\eqreff{e:tu=s}, $S=TU$, $[T,U]=1$, $T$ and $U$ are strongly closed, and 
$\cale$ and $\cald$ commute, so $S=T\times U$, and 
$\calf=\cale\times\cald$ by Lemma \reff{l:F=ExD}. By construction, 
$f|_T\in\Aut(\cale)$ and $f|_U$ is nilpotent. This proves 
the existence statement in the finite case, and uniqueness will be shown 
below. 

\smallskip

\noindent\boldd{In the general case,} set $Z=T_+\cap U$. Since $f(T_+)=T_+$ 
and $f(U)\le U$, we have $f(Z)\le Z$. Also, $Z\le f^n(S)\cap\psi_n(S)\le 
Z(\calf)$ by Proposition \reff{p:normal.end.2}(a). For $z\in Z$, $z=f(t)$ 
for some $t\in T_+$, $f^i(z)=f^{i+1}(t)=1$ for some $i$, and hence $t\in 
T_+\cap U=Z$. Thus $f|_Z\in\End(Z)$ is surjective and locally nilpotent. So 
$Z$ is connected by Lemma \reff{l:End(S)-2}. If $f$ is surjective, then 
$T_+=S$, and so $U=Z$ is connected. If $|Z(\calf)|<\infty$, then $Z=1$. 

\smallskip

\noindent\textbf{Case 1: } If $Z=1$ (in particular, if 
$|Z(\calf)|<\infty$), then set $T=T_+$ and $\cale=\cale_+$. By 
\eqreff{e:tu=s} together with Lemma \reff{l:F=ExD}, we have $S=T\times U$, 
$U=\Ker(f^n)$, and $\calf=\cale\times\cald$. Thus $f|_U$ is nilpotent. 
Also, $\Ker(f|_{T_+})=1$, so $f|_{T_+}\in\Aut(\cale_+)$ by Lemma 
\reff{l:inj+surj=iso}(b). 


\smallskip

\noindent\textbf{Case 2: } If $Z\ne1$, set $f_0=f|_{T_+}\in\End(\cale_+)$. 
Then $f_0$ is surjective by definition of $T_+$ and $\cale_+$. Since there 
is $\chi\in\End(\calf)$ such that $f+\chi=\Id_\calf$, we have 
$f_0+\chi|_{T_+}=\Id_{\cale_+}$ where $\chi(T_+)\le T_+$, and so 
$f_0\in\End^N(\cale_+)$. 

By Lemma \reff{l:[f,S]<Z(F)}(a) and since $f_0$ is normal and surjective, 
$\Ker(f_0^i)\le[f_0^i,T_+]\le Z(\cale_+)$ for each $i$ and 
$f|_{\foc(\cale_+)}=\Id$. Hence $Z\le Z(\cale_+)$ and 
$Z\cap\foc(\cale_+)=1$. 

Set $\4{S}=T_+/\foc(\cale_+)$, and let $\4{f}\in\End(\4{S})$ be the 
endomorphism induced by $f_0$. Then $\4{S}$ is abelian and $\4{f}$ is 
surjective, so by Lemma \reff{l:End(S)-2}, there is a unique factorization 
$\4{S}=\4T\times\4Z$ such that $\4{f}|_{\4T}\in\Aut(\4T)$ and 
$\4{f}|_{\4Z}\in\End(\4Z)$ is locally nilpotent. 

Let $T\le T_+$ be such that $T\ge\foc(\cale_+)$ and $T/\foc(\cale_+)=\4T$. 
Then $f|_T\in\Aut(T)$ since $f_0=f|_{T_+}$ induces the identity on 
$\foc(\cale_+)$ and an automorphism of the quotient. Hence 
$T\cap\Ker(f^i)=1$ for all $i$, so $T\cap U=1$. For each $x\in T_+$, since 
$\4{S}=\4T\4Z$, there are $x_1,x_2\in T_+$ such that $x=x_1x_2$, $x_1\in 
T$, and $f_0^i(x_2)=y\in\foc(\cale_+)$ for some $i\ge1$. Since 
$f|_{\foc(\cale_+)}=\Id$, we have $y^{-1}x_2\in\Ker(f_0^i)\le Z$, and so 
$x=(x_1y)(y^{-1}x_2)\in TZ$. Thus $T_+=TZ$, so $S=T_+U=TZU=TU$ since $Z\le 
U$, and $S=T\times U$ since $[T,U]\le[T_+,U]=1$.

Set $\cale=\calf|_{\le T}$. By Lemma \reff{l:Z+foc}(b), $T$ is strongly 
closed in $\cale_+$ since it contains $\foc(\cale_+)$. Since 
$\cale_+=\calf|_{\le T_+}$ where $T_+$ is strongly closed in $\calf$, this 
implies that $T$ is strongly closed in $\calf$. We already saw that $U$ is 
strongly closed in $\calf$, and $\cale$ commutes with $\cald$ in $\calf$ 
since $\cale_+$ commutes with $\cald$. So $\calf=\cale\times\cald$ by Lemma 
\reff{l:F=ExD}. Also, $\Ker(f|_T)=1$, so $f|_T\in\Aut(\cale)$ by Lemma 
\reff{l:inj+surj=iso}(b), finishing the proof of existence of such a 
factorization. 

\smallskip

\noindent\textbf{Uniqueness: } \new{Assume} $\calf=\cale^*\times\cald^*$ 
over $S=T^*\times U^*$ is a second factorization, \new{where 
$f|_{T^*}\in\Aut^N(\cale^*)$, and $f|_{U^*}\in\End^N(\cald^*)$ and is locally 
nilpotent. Then} $U^*=\bigcup_{i=1}^\infty\Ker(f^i)=U$, and 
$T^*\le\bigcap_{i=1}^\infty f^i(S)=T_+$. Also, $\foc(\cale_+)\le T^*$ since 
$f|_{\foc(\cale_+)}=\Id$. So $T_+/\foc(\cale_+)=(T^*/\foc(\cale_+))\times 
V$ where $V$ is the image of $(U^*\cap T_+)$ in the quotient, $f$ induces 
an isomorphism on the first factor, and induces a locally nilpotent 
endomorphism on the second factor. By the uniqueness statement in Lemma 
\reff{l:End(S)-2}, applied again with $T_+/\foc(\cale_+)$ in the role of 
$S$, we have $T^*/\foc(\cale_+)=\4T=T/\foc(\cale_+)$ and hence $T^*=T$. 
Also, $\cale^*=\calf|_{\le T^*}=\cale$ and $\cald^*=\calf|_{\le 
U^*}=\cald$, and so the factorization is unique. 
\end{proof}

\nnewpage

\section{A Krull-Remak-Schmidt theorem for fusion systems} 

By analogy with the Krull-Remak-Schmidt theorem for groups, we look at a 
slightly more general version of Theorem \reff{ThA} where we assume all 
direct factors are invariant under a given group $\Omega$ of automorphisms. 
If one ignores the next paragraph, and takes $\Omega=1$ in Lemma 
\reff{l:normal.end.} and Theorem \reff{t:KRS}, then one gets Theorem 
\reff{ThA} as stated in the introduction.

Let $\calf$ be a saturated fusion system over a discrete $p$-toral group 
$S$, and let $\Omega\le\Aut(\calf)$ be a group of automorphisms. We say that 
\begin{itemize} 

\item a fusion subsystem $\cale\le\calf$ is \emph{$\Omega$-invariant} if 
$\5\omega(\cale)=\cale$ for each $\omega\in\Omega$; 

\item an $\Omega$-invariant subsystem $\cale\le\calf$ is 
\emph{$\Omega$-indecomposable} if there are no proper $\Omega$-invariant 
subsystems $\cale_1,\cale_2<\cale$ such that $\cale=\cale_1\times\cale_2$; 
and 

\item an endomorphism of $\calf$ is \emph{$\Omega$-normal} if it is normal 
and commutes in $\End(\calf)$ with all $\omega\in\Omega$. 

\end{itemize}
\noindent The sets of all $\Omega$-normal endomorphisms and 
$\Omega$-normal automorphisms of $\calf$ form a monoid and a group, 
respectively, which we denote 
	\[ \End^\Omega(\calf)=C_{\End^N(\calf)}(\Omega) 
	\qquad\textup{and}\qquad
	\Aut^\Omega(\calf) = C_{\Aut^N(\calf)}(\Omega). \]
Here, $\Omega$ acts on $\End^N(\calf)$ by conjugation. Note that sums of 
$\Omega$-normal endomorphisms are $\Omega$-normal when the the hypotheses 
of Lemma \reff{l:normal.end.1}(b) hold.

\begin{Lem} \label{l:normal.end.}
Let $\calf$ be a saturated fusion system over a discrete $p$-toral group 
$S$. Fix a subgroup $\Omega\le\Aut(\calf)$, and assume that $\calf$ is 
$\Omega$-indecomposable. Then 
\begin{enuma} 

\item each $\Omega$-normal endomorphism $f\in\End^\Omega(\calf)$ is either 
nilpotent or an isomorphism, or possibly (if $Z(\calf)$ is infinite) 
locally nilpotent; and 

\item if $f_1,\dots,f_k\in\End^\Omega(\calf)$ are summable, and 
$f_1+\dots+f_k\in\Aut(\calf)$, then $f_i$ is an automorphism for some 
$i=1,\dots,k$.

\end{enuma}
\end{Lem}

\begin{proof} \textbf{(a) } By Proposition \reff{p:F=ExD}, there is a 
unique factorization $\calf=\cale\times\cald$ over $S=T\times U$ such that 
$f|_T\in\Aut(\cale)$ and $f|_U\in\End(\cald)$ is locally nilpotent. For 
each $\omega\in\Omega$, since $f$ commutes with $\omega$, we get another 
factorization $\calf=\5\omega(\cale)\times\5\omega(\cald)$ with the same 
properties. Hence $\5\omega(\cale)=\cale$ and $\5\omega(\cald)=\cald$ by 
the uniqueness of factorization, so $\cale$ and $\cald$ are 
$\Omega$-invariant. Since $\calf$ is $\Omega$-indecomposable, $\calf=\cale$ 
or $\calf=\cald$, and so $f\in\Aut(\calf)$ or $f$ is locally nilpotent. By 
Proposition \reff{p:F=ExD} again, if $f$ is locally nilpotent and 
$|Z(\calf)|<\infty$, then $f$ is nilpotent. 

\smallskip

\noindent\textbf{(b) } \new{By induction,} it suffices to prove this when 
$k=2$. Assume otherwise: let $f_1,f_2\in\End^\Omega(\calf)$ be summable, 
and such that $f_1+f_2$ is an automorphism but neither $f_1$ nor $f_2$ is 
one. We claim that this is impossible. Set $\alpha=f_1+f_2\in\Aut(\calf)$: 
then $\alpha\omega=\omega\alpha$ for all $\omega\in\Omega$ (but we do not 
assume $\alpha$ is normal). Set $f'_1=f_1\alpha^{-1}$ and 
$f'_2=f_2\alpha^{-1}$. By Lemma \reff{l:(f1+f2)(f3+f4)}, $f'_1$ and $f'_2$ 
are summable and $f'_1+f'_2=\Id_\calf$. So $f'_1$ and $f'_2$ are normal, 
and $f'_1,f'_2\in\End^\Omega(\calf)$ since the $f_i$ and $\alpha$ all 
commute with $\Omega$. Upon replacing $f_i$ by $f'_i$, we can now assume 
that $f_1+f_2=\Id_\calf$. Since neither $f_1$ nor $f_2$ is an isomorphism, 
they are both locally nilpotent by (a).

Thus for $1\ne x\in S$, there are $n,m\ge1$ such that 
$f_1^n(x)=1=f_2^m(x)$. Also, $f_1f_2=f_2f_1$ by Proposition 
\reff{p:normal.end.2}(a), so $(f_1+f_2)^{n+m}(x)=1$, which is impossible 
when $f_1+f_2=\Id_\calf$. 
\end{proof}

We are now ready to prove Theorem \reff{ThA}, in the following, stronger 
version. As noted above, Theorem \reff{ThA} (the uniqueness part) is the 
special case of Theorem \reff{t:KRS} where $\Omega=1$.

\begin{Thm} \label{t:KRS}
Let $\calf$ be a saturated fusion system over a discrete $p$-toral group 
$S$, and fix $\Omega\le\Aut(\calf)$. Let 
$\cale_1,\dots,\cale_k$ and $\cale_1^*,\dots,\cale_m^*$ be 
$\Omega$-indecomposable $\Omega$-invariant fusion subsystems of $\calf$ 
such that 
	\[ \calf = \xxx\cale = \xxx[m]{\cale^*}. \]
Then $k=m$, and there are $\alpha\in\Aut^\Omega(\calf)$ and 
$\sigma\in\Sigma_k$ such that $\alpha(\cale_i)=\cale^*_{\sigma(i)}$ for 
each $i$.
\end{Thm}

\begin{proof} Set $\cale^*=\cale_1^*$ and 
$\cald=\cale_2^*\times\cdots\times\cale_m^*$; thus 
$\calf=\cale^*\times\cald$. Let $T_1,\dots,T_k,T^*,U\le S$ be such that 
$\cale_i$, $\cale^*$, and $\cald$ are fusion subsystems over $T_i$, $T^*$, 
and $U$, respectively. 

Let $f_1,\dots,f_k,g,g'\in\End^\Omega(\calf)$ 
be the projections to $\cale_1,\dots,\cale_k,\cale^*,\cald$, respectively. 
Thus $f_1+\dots+f_k=\Id_\calf=g+g'$. Also, 
	\[ g|_{T^*} = g\circ(f_1+\dots+f_k)|_{T^*} = 
	gf_1|_{T^*}+\dots+gf_k|_{T^*}  \]
by Lemma \reff{l:(f1+f2)(f3+f4)}, and we regard this as a sum of endomorphisms of 
$\cale^*$. Since $g|_{T^*}$ is an automorphism, at least one of the 
summands is an automorphism by Lemma \reff{l:normal.end.}(b) (applied with 
$\{\omega|_{T^*}\,|\,\omega\in\Omega\}\le\Aut(\cale^*)$ in the place of 
$\Omega$). 

Let $j$ be such that $gf_j|_{T^*}\in\Aut(\cale^*)$. Thus 
$f_j|_{T^*}$ is an injective morphism from $\cale^*$ to $\cale_j$, and 
$g$ restricts to a surjection from $\gen{\5f_j(\cale^*)}$ onto $\cale^*$. 

Consider $f_jg|_{T_j}\in\End(\cale_j)$. For each $n\ge1$, $(f_jg)^n|_{T_j} 
= (f_j|_{T^*})((gf_j)^{n-1}|_{T^*})(g|_{T_j})$, and since 
$gf_j|_{T^*}\in\Aut(\cale^*)$ and $f_j|_{T^*}$ is injective, we see that 
$\Ker((f_jg)^n|_{T_j})=\Ker(g|_{T_j})$. Since $g|_{T_j}$ is nontrivial, 
this shows that $f_jg|_{T_j}$ is not locally nilpotent, and hence is an 
automorphism of $\cale_j$ by Lemma \reff{l:normal.end.}(a). Then 
$g|_{T_j}\in\Iso(\cale_j,\cale^*)$ and 
$f_j|_{T^*}\in\Iso(\cale^*,\cale_j)$. 

Now, $f_jg$ and $g'$ are both $\Omega$-normal endomorphisms of $\calf$, and 
$(f_jg)g'=\zero{\calf}$ since $gg'=\zero\calf$. So by Lemma 
\reff{l:normal.end.1}(b), $f_jg$ and $g'$ are summable and $f_jg+g'$ is 
$\Omega$-normal. Set $h_1=f_jg+g'\in\End^\Omega(\calf)$. Then 
$h_1|_{T^*}=f_jg|_{T^*}=f_j|_{T^*}$ sends $\cale^*$ isomorphically to 
$\cale_j$, while $h_1|_U=\Id_\cald$. 

Assume $x\in\Ker(h_1)\le S$. Then $1=gh_1(x)=(gf_jg+gg')(x)=gf_j(g(x))$, 
and $g(x)=1$ since $g(S)=T^*$ and $gf_j|_{T^*}$ is an automorphism. Hence 
$1=h_1(x)=(f_jg+g')(x)=g'(x)$, so $x=(g+g')(x)=g(x)g'(x)=1$. This proves 
that $\Ker(h_1)=1$, and hence (since $|h_1(S)|=|S|$) that $h_1(S)=S$. Then 
$\Im(h_1)=\calf|_{\le S}=\calf$ by Proposition \reff{p:normal.end.2}(d), so 
$h_1\in\Aut^\Omega(\calf)$ by Lemma \reff{l:inj+surj=iso}(b).

We have now constructed $h_1\in\Aut^\Omega(\calf)$ that sends 
$\cale^*=\cale_1^*$ isomorphically to $\cale_j$ \new{(where $j$ is as 
above)}, and is the identity on $\cale_i^*$ for $2\le i\le m$. In 
particular, $\calf=\cale_j\times\cale_2^*\times\dots\times\cale_m^*$. Upon 
repeating this construction, but with $\cale^*=\cale_2^*$ and 
$\cald=\cale_j\times\cale_3^*\times\dots\times\cale_m^*$, we obtain 
$h_2\in\Aut^\Omega(\calf)$ that sends $\cale_2^*$ isomorphically to 
$\cale_{j_2}$ \new{for some $j_2\in\{1,\dots,k\}$,} and $\cald$ 
to itself via the identity. \new{Also, $j_2\ne j$ since $h_2$ is injective 
($\cale_{j_2}=h_2(\cale_2^*)\ne h_2(\cale_j)=\cale_j$).}

Upon continuing this process, we obtain $\Omega$-normal automorphisms 
$h_1,\dots,h_m$ of $\calf$ such that for each $i=1,\dots,m$, 
$h=h_m\circ\cdots\circ h_1$ sends $\cale_i^*$ isomorphically to 
$\cale_{j_i}$ for some $j_i\in\{1,\dots,k\}$. The $j_i$ are distinct since 
$h$ is injective, and $\{j_1,\dots,j_m\}=\{1,\dots,k\}$ since $h$ is 
an isomorphism. So $m=k$, and $h\in\Aut^\Omega(\calf)$ sends each $\cale_i^*$ to 
some $\cale_j$. 
\end{proof}

The first corollary is a special case of Theorem \ref{t:KRS}.

\begin{Cor} \label{c:KRS}
If $\calf$ is a saturated fusion system over a discrete $p$-toral group 
$S$ such that either $Z(\calf)=1$ or $\foc(\calf)=S$, then $\calf$ factors 
as a product of indecomposable fusion subsystems in a unique way. Thus 
$\calf$ is the direct product of all of its indecomposable direct 
factors. 
\end{Cor}

\begin{proof} If $Z(\calf)=1$ or $\foc(\calf)=S$, then 
$\Aut^N(\calf)=\{\Id_\calf\}$ by Lemma \reff{l:[f,S]<Z(F)}. So the result 
follows immediately from Theorem \reff{t:KRS} (applied with $\Omega=1$).
\end{proof}

Our original interest in this problem arose from trying to describe 
automorphism groups of fusion systems in terms of those of their 
indecomposable factors. Our last corollary is a very simple application of 
this type. (Compare it with Propositions 3.4 and 3.8 in \cite{BMOR}.)

\begin{Cor} \label{c:Aut(F1xF2)} 
Let $\calf$ be a saturated fusion system over a discrete $p$-toral group 
$S$ such that either $Z(\calf)=1$ or $\foc(\calf)=S$. Assume 
$\calf=\xxx\cale$ where each $\cale_i$ is an indecomposable 
fusion subsystem over $T_i\nsg S$, and set 
	\[ \Aut^0(\calf) = 
	\bigl\{ \alpha\in\Aut(\calf) \,\big|\, \alpha(T_i)=T_i ~\textup{for 
	each $1\le i\le k$} \bigr\}. \]
Let $\Gamma\le\Sigma_k$ be the subgroup of all permutations 
$\sigma$ such that $\cale_{\sigma(i)}\cong\cale_i$ for each $i$. Then 
$\Aut^0(\calf)\cong\prod_{i=1}^k\Aut(\cale_i)$ and is normal in 
$\Aut(\calf)$, and there is a subgroup $K\le\Aut(\calf)$ such that 
$K\cong\Gamma\cong\Aut(\calf)/\Aut^0(\calf)$ and 
$\Aut(\calf)=\Aut^0(\calf)\rtimes K$.
\end{Cor}

\begin{proof} The isomorphism 
$\Aut^0(\calf)\cong\prod_{i=1}^k\Aut(\cale_i)$ is clear. By Corollary 
\reff{c:KRS}, for each $\alpha\in\Aut(\calf)$, there is $\sigma\in\Sigma_k$ 
such that $\alpha(T_i)=T_{\sigma(i)}$ for each $i$, and $\sigma\in\Gamma$ 
since $\cale_{\sigma(i)}=\5\alpha(\cale_i)\cong \cale_i$ for each $i$. This 
defines a homomorphism $\rho\:\Aut(\calf)\too\Gamma$ with kernel 
$\Aut^0(\calf)$. 

To see that $\Aut(\calf)$ is a semidirect product and $\rho$ is surjective, 
choose isomorphisms $\beta_{i,j}\in\Mor(\cale_i,\cale_j)$ for each $1\le 
i<j\le k$ such that $\cale_i\cong\cale_j$, in such a way that 
$\beta_{h,j}=\beta_{i,j}\circ\beta_{h,i}$ whenever $h<i<j$ are such that 
$\cale_h\cong\cale_i\cong\cale_{j}$. Set $\beta_{i,i}=\Id_{\cale_i}$ for 
all $i$, and let $\beta_{j,i}=\beta_{i,j}^{-1}$ whenever $\beta_{i,j}$ is 
defined. Set 
	\[ K = \bigl\{\alpha\in\Aut(\calf) \,\big|\, 
	\textup{$\forall\, i=1,\dots,k$,}~ 
	\alpha|_{\cale_i}=\beta_{i,j} ~\textup{where $j=\rho(\alpha)(i)$} 
	\bigr\}. \]
Then $\rho|_K$ is an isomorphism from $K$ to $\Gamma$, and hence 
$\Aut(\calf)=\Aut^0(\calf)\rtimes K$.
\end{proof}

\begin{New} 

We now return to the question raised in the introduction: that of how 
easily Theorem \ref{t:KRS}, after restriction to realizable fusion systems 
over finite $p$-groups (i.e., fusion systems realized by finite groups), 
can be derived from the Krull-Remak-Schmidt theorem for finite groups. The 
quick answer to this seems to be that in most cases (at least), it requires 
tools much more sophisticated than those used here to prove Theorem 
\ref{t:KRS}. 

The simplest case seems to be that when $p=2$ and $O^{2'}(\calf)=\calf$. 
Here, one can apply the following lemma, which is based on a theorem of 
Goldschmidt.

\begin{Lem} \label{l:G=prod}
Let $\calf$ be a realizable fusion system over a finite $2$-group $S$ such 
that $O^{p'}(\calf)=\calf$. Then there is a finite group $G$ such that 
$O_{2'}(G)=1$, $O^{2'}(G)=G$, $S\in\syl2G$ and $\calf=\calf_S(G)$. For each 
such $G$, if $\calf=\xxx\cale$ is a factorization of $\calf$ over 
$S=\xxx{T}$, then there are subgroups $H_1,\dots,H_k\le G$ such that 
$T_i\in\syl2{H_i}$ and $\cale_i=\calf_{T_i}(H_i)$ for each $i$, and such 
that $G=\xxx{H}$.
\end{Lem}

\begin{proof} If $\Gamma$ is an arbitrary finite group that realizes 
$\calf$, then $\calf$ is also realized by $\Gamma/O_{2'}(\Gamma)$, and by 
$G\defeq O^{2'}(\Gamma/O_{2'}(\Gamma))$ since $O^{2'}(\calf)=\calf$. 

For each $1\le i\le k$, let $H_i$ be the normal closure of 
$T_i$ in $G$. By \cite[Theorem A]{Goldschmidt} and since $O_{2'}(G)=1$, the 
$H_i$ commute pairwise in $G$. Hence there is a homomorphism $I$ from $\xxx{H}$ 
to $G$, its kernel has odd order since $S=\xxx{T}$, and $I$ is injective 
since $O_{2'}(H_i)=1$ for each $i$ (since $O_{2'}(G)=1$). 

By construction, $\Im(I)$ is the normal closure of $S$ in $G$, hence equal 
to $G$ since $O^{2'}(G)=G$. Thus $G=\xxx{H}$. Also, $T_i\le H_i$ for each 
$i$ by construction, and $T_i\in\syl2{H_i}$ since $S\in\syl2G$. Also, 
$\calf_{T_i}(H_i)\le\cale_i$ for each $i$, with equality since 
$\calf=\prod_{i=1}^k\cale_i$ and 
$\calf=\calf_S(G)=\prod_{i=1}^k\calf_{T_i}(H_i)$.
\end{proof}

Thus under these assumptions ($p=2$ and $O^{2'}(\calf)=\calf$), two 
distinct direct factor decompositions of $\calf$ give rise to two distinct 
factorizations of $G$, and the factors of $G$ are indecomposable if those 
of $\calf$ are. Since the two factorizations of $G$ are linked by a normal 
automorphism, so are those of $\calf$.

This argument can be extended to one that applies to arbitrary realizable 
fusion systems over finite $2$-groups, but as far as we can see, only with 
the help of additional tools such as the notion of tameness of a fusion 
system and \cite[Theorem C]{BMOR} (all realizable fusion systems are tame). 
When $p$ is odd, a result analogous to Lemma \ref{l:G=prod} can be shown 
using Theorem 1.2 in \cite{FF}. The proofs of both of these last two 
mentioned theorems require the classification of finite simple groups.


\end{New}

\bigskip


\end{document}